\pdfoutput=1
\documentclass[microtype]{gtpart}
\usepackage{graphicx}
\usepackage[mathscr]{eucal}
\usepackage{amssymb}
\usepackage[dvipsnames]{xcolor}
\usepackage{enumerate, cite}
\usepackage[margin=1.3in]{geometry}
\usepackage{hyperref}

\newcommand{\br}{\mathbb{R}}

\newcommand{\bz}{\mathbb Z}
\newcommand{\bn}{\mathbb N}

\newcommand{\bS}{\mathbb S}

\newcommand{\cE}{\mathcal E}

\newcommand{\ve}{\varepsilon}
\newcommand{\vp}{\varphi}

\newcommand{\ssm}{\smallsetminus}

\DeclareMathOperator{\mcg}{MCG}

\DeclareMathOperator{\Ends}{\cE}

\DeclareMathOperator{\Homeo}{Homeo}

\DeclareMathOperator{\supp}{supp}

\renewcommand{\co}{\colon\thinspace}

\newtheorem{Thm}{Theorem}[section]
\newtheorem{Thm*}{Theorem}
\newtheorem{Prop}[Thm]{Proposition}
\newtheorem{Lem}[Thm]{Lemma}
\newtheorem{Cor}[Thm]{Corollary}
\newtheorem{Cor*}[Thm*]{Corollary}

\newtheorem{Question}[Thm]{Question}

\theoremstyle{definition}
\newtheorem{Def}[Thm]{Definition}

\newtheorem*{Ex*}{Examples}

\newtheorem{Problem}[Thm]{Problem}

\numberwithin{equation}{section}

\title[Homeomorphism groups of telescoping 2-manifolds]{Homeomorphism groups of telescoping 2-manifolds\\are strongly distorted}

\author{Nicholas G. Vlamis}
\address{Department of mathematics \\ CUNY Graduate Center \\ New York, NY 10016, and \newline Department of Mathematics \\ CUNY Queens College \\ Flushing, NY 11367}
\email{nicholas.vlamis@qc.cuny.edu}

\begin{document}  

\begin{abstract}
Building on the work of Mann and Rafi, we introduce an expanded definition of a telescoping 2-manifold and proceed to study the homeomorphism group of a telescoping 2-manifold.
Our main result shows that it is strongly distorted.
We then give a simple description of its commutator subgroup, which is index one, two, or four depending on the topology of the manifold.
Moreover, we show its commutator subgroup is uniformly perfect with commutator width at most two, and we give a family of uniform normal generators for its commutator subgroup. 
As a consequence of the latter result, we show that for an (orientable) telescoping 2-manifold,  every (orientation-preserving) homeomorphism can be expressed as a product of at most 17 conjugate involutions. 
Finally, we provide analogous statements for mapping class groups.
\end{abstract}

\maketitle

\vspace{-0.5in}


\section{Introduction}

The homeomorphism groups of spheres, Euclidean spaces, and annuli share many boundedness properties.  
We recount several such properties of particular relevance to this article: 
Calegari--Freedman \cite{CalegariDistortion} showed that \( \Homeo(\bS^n) \) is strongly distorted, implying that any left-invariant metric on one of these groups has bounded diameter; Le Roux--Mann \cite{LeRouxStrong} established strong distortion for \( \Homeo(\br^n) \) and \( \Homeo(\bS^n \times \br) \); Anderson \cite{AndersonAlgebraic} showed that \( \Homeo^+(\bS^n) \) is uniformly simple\footnote{This citation is misleading as Anderson established this result only for dimensions two and three; however, his techniques can be extended to higher dimensions using the fragmentation lemma of Edwards--Kirby \cite{EdwardsDeformations}.}, more precisely, he showed every element can be expressed as a product of at most eight conjugates of any given nontrivial element and its inverse; 
and Le Roux--Mann \cite{MannAutomatic1} showed that the group of germs of orientation-preserving homeomorphisms at each end of \( \br^n \) and \( \bS^1 \times \br^n \) is uniformly simple (see also Ling \cite{LingAlgebraic}).

By extracting the core, shared topological aspects of the 2-sphere, plane, and annulus that lead to the boundedness properties of their homeomorphism groups described above, we establish a class of 2-manifolds\footnote{We follow the convention that manifolds do not have boundary.}, namely the telescoping 2-manifolds, whose homeomorphism groups exhibit similar boundedness properties. 
The notion of a telescoping 2-manifold was first introduced by Mann--Rafi \cite{MannLarge}, but the definition we give here is broader than the original. 

We will provide more motivation and context, as well as applications and further questions, after introducing and summarizing our main results.

\begin{figure}
\centering
\includegraphics[scale=0.75]{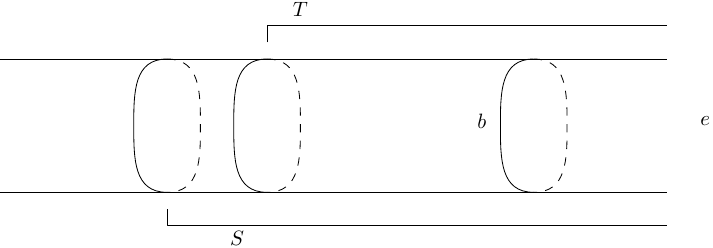}
\caption{A visual of the definition of a telescope with \( M = \mathbb S^1 \times \br \).}
\label{fig:telescope}
\end{figure}

\begin{Def}[Telescoping 2-manifold]
\label{def:telescope}
    Let \( M \) be a 2-manifold, and let \( T \) be a subsurface of \( M \) with connected compact boundary.
    We say \( T \) is a \emph{telescope} if either \( T \) is homeomorphic to the closed unit disk\footnote{Separating out the case of the closed disk is a bit artificial, as the definition can be rephrased using the Freudenthal compactification of \( M \) in such a way that can combine the compact and non-compact cases.} \( \mathbb D \) or there exists an end \( e \) of \( M \) seen by \( T \) and a subsurface \( S \)  with connected compact boundary such that
    	\begin{enumerate}[(i)]
    		\item \( T \subset S \) and  \( \partial T \) separates \( \partial S \) from \( e \), 
    		\item if the \( \Homeo(M) \)-orbit of \( e \) has cardinality greater than one, then \( S \) induces a nontrivial partition of its orbit\footnote{This technical condition is elaborated on in Section~\ref{sec:examples}.}, and
    		\item given any separating simple closed curve \( b \) in \( T \) separating \( \partial T \) from \( e \), there exists a homeomorphism \( M \to M \) that restricts to the identity on \( M \ssm S \) and  maps \( b \) onto \( \partial T \).
    	\end{enumerate}
    We call \( e \) a \emph{defining end} of \( T \), and we say  \( S \) is an \emph{extension} of \( T \). 
    If in addition, there exists a homeomorphism \( M \to M \) mapping \( M \ssm T \) into \( T \), then \( T \) is called a \emph{maximal telescope}. 
    A 2-manifold is \emph{telescoping} if it contains a maximal telescope.
\end{Def}

By the Sch\"onflies theorem and the classification of surfaces, \( \mathbb S^2 \), \( \mathbb R^2 \), and \( \mathbb S^1 \times \br \) (see Figure~\ref{fig:telescope}) are telescoping.
In fact, these are all of the finite-type telescoping 2-manifolds:
It is an exercise in the definition to see that there are no telescoping 2-manifolds with finite positive genus; in particular,  a telescoping 2-manifold is either planar or it has infinite genus.
Similarly, if a telescoping 2-manifold has more than two ends, it must have infinitely many. 
We describe an uncountable family of  infinite-type telescoping 2-manifolds in Section~\ref{sec:examples}.

A group \( G \) is \emph{strongly distorted} if there exists \( m \in \bn \) and \( \{w_n\}_{n\in\bn} \subset \bn \) such that for any sequence \( \{ g_n \}_{n\in\bn} \) in \( G \) there exists \( \mathcal S \subset G \) such that  \( g_n \in \mathcal S^{w_n} \) and \( |\mathcal S| = m \). 
Here, \( \mathcal S^k \) denotes the set \( \{ s_1s_2 \cdots s_k : s_1, \ldots, s_k \in \mathcal S \} \). 
The name stems from the fact that every infinite-order element of a strongly distorted group is arbitrarily distorted in the standard sense.
Given the discussion thus far, we know that the homeomorphism groups of the finite-type telescoping 2-manifolds---\( \mathbb S^2 \), \( \br^2 \), and \( \mathbb S^1 \times \br \)---are strongly distorted, and our main result extends this to all telescoping 2-manifolds.

\begin{Thm*}
\label{thm:distortion}
The homeomorphism group of a telescoping 2-manifold is strongly distorted. 
\end{Thm*}

For the entirety of the article, including in proving Theorem~\ref{thm:distortion}, we work with a specific finite-index subgroup of the homeomorphism group of a telescoping 2-manifold, which we now introduce.
The choice of subgroup depends on the structure of the maximal ends of the manifold.
The notion of a maximal end was introduced by Mann--Rafi \cite{MannLarge}; we give the definition in Section~\ref{sec:prelim}, and for now, simply note that the set of maximal ends of a telescoping 2-manifold is either a singleton, a doubleton, or a perfect set\footnote{Recall that the empty set is a perfect set, so the 2-sphere falls into the third category.}.

\textbf{Notation.}
Let \( M \) be a telescoping 2-manifold.
    \begin{itemize}
        \item If \( M \) has a perfect set of maximal ends, let \( H(M) = \Homeo^+(M) \) if \( M \) is orientable, otherwise let \( H(M) = \Homeo(M) \).
        \item If \( M \) is orientable and has a finite number of maximal ends, let \( H(M) \) denote the group of orientation-preserving homeomorphisms of \( M \)  that fix each maximal end of \( M \).
        \item If \( M \) is not orientable and has a finite number of maximal ends, let \( H(M) \) denote the subgroup of \( \Homeo(M) \) consisting of homeomorphisms fixing each maximal end of \( M \). 
    \end{itemize}
Note that if \( M \) is one of \( \bS^2 \), \( \br^2 \), or \( \bS^1 \times \br \), then \( H(M) = \Homeo_0(M) \),  the connected component of the identity.

As strong distortion is inherited by finite extensions, to establish Theorem~\ref{thm:distortion} it is enough to show that \( H(M) \) is strongly distorted, which  is accomplished in Theorem~\ref{thm:distortion2} in the case of a finite number of maximal ends,  with explicit constants. 
The case of a perfect set of maximal ends was established in \cite[Theorem~5.9.2]{VlamisHomeomorphism}.

Recall the  \emph{commutator subgroup} of a group \( G \), denoted \( [G,G] \), is the subgroup generated by \emph{commutators}, that is, elements of the form \( [a,b] = aba^{-1}b^{-1} \) for \( a,b \in G \). 
Given \( p \in \bn \), a group \( G \) is \emph{uniformly perfect with commutator width \( p \)} if every element of \( G \) can be expressed as a product of  \( p \) commutators and if there exists an element that cannot be expressed as a product of \( p-1 \) commutators.

\begin{Thm*}
\label{thm:perfect}
If \( M \) is a telescoping 2-manifold, then \( H(M) = [\Homeo(M), \Homeo(M)] \) and \( H(M) \) is uniformly perfect with commutator width at most two. 

\end{Thm*}

It is clear that the commutator subgroup of \( \Homeo(M) \) is contained in \( H(M) \), and so the content of Theorem~\ref{thm:perfect} is in showing that \( H(M) \) is uniformly perfect; this is established in Theorem~\ref{thm:uniform2} in the case of a finite number of maximal ends, and the case of a perfect set of maximal ends was established in \cite[Theorem~A]{MalesteinSelf} (see also \cite[Theorem~5.8.1]{VlamisHomeomorphism}). 
Theorem~\ref{thm:perfect} for the 2-sphere is due to Anderson \cite{AndersonAlgebraic}; it is not clear how to attribute the cases of \( \br^2 \) and \( \mathbb S^1 \times \br \).
We note that Tsuboi \cite{TsuboiHomeomorphism} proved that the commutator width of \( H(\mathbb S^n) \) is one, and Bhat and the author \cite{BhatOrientation} showed that the commutator width of \( H(\br^n) \) and \( H(\bS^n \times \br) \) are also one. 
As an immediate corollary, we compute the abelianization of \( \Homeo(M) \), that is, the group \( \Homeo(M)/[\Homeo(M), \Homeo(M)] \). 

\begin{Cor} 
Let \( M \) be a telescoping 2-manifold.
\begin{enumerate}[(1)]
	\item Suppose \( M \) is orientable. 
		If \( M \) has two maximal ends, then the abelianization of \( \Homeo(M) \) is \( \bz/2\bz \times \bz/2\bz \); otherwise, the abelianization is \( \bz/2\bz \).
	\item Suppose \( M \) is not orientable.
		If \( M \) has two maximal ends, then the abelianization of \( \Homeo(M) \) is \( \bz/2\bz \); otherwise, the abelianization is trivial. \qed
\end{enumerate}
\end{Cor}

Though presented second, Theorem~\ref{thm:perfect} is motivation for Theorem~\ref{thm:distortion}.
It is a standard fact that the image of the set of commutators under a quasimorphism has bounded diameter (see \cite{Calegariscl}).
As noted in \cite{CalegariDistortion}, a standard trick for finding undistorted elements in a group is to find unbounded quasimorphisms, and hence, given a large group in which every quasimorphism is bounded (e.g., a uniformly perfect group), it is natural to ask if every element in the group is arbitrarily distorted.
This is indeed the case for telescoping 2-manifolds by Theorem~\ref{thm:distortion}.

The implication of bounded quasimorphisms also shows that Theorem~\ref{thm:perfect} is a strengthening of a result of Lanier and the author \cite[Corollary~1.9]{LanierHomeomorphism} when restricted to telescoping 2-manifolds with a finite number of maximal ends.
In the parlance of \cite{LanierHomeomorphism}, telescoping 2-manifolds with a finite number of maximal ends form a subclass of the class of spacious 2-manifolds.
Lanier and the author showed that for a spacious 2-manifold \( M \), the group \( H(M) \) has a dense conjugacy class, and hence as a consequence, every continuous quasimorphism of \( \Homeo(M) \) is bounded.

An element \( g \) of a group \( G \) \emph{uniformly normally generates} the group if there exists \( p \in \bn \) such that every element in \( G \) is a product of at most \( p \) conjugates of \( g \) and \( g^{-1} \). 
In Definition~\ref{def:dilation} below, we introduce the notion of a \emph{dilation at infinity}, which is meant to describe homeomorphisms that have a contracting or expanding behavior at each maximal end. 
For example, given \( k \in (0, \infty) \), the map \( \br^2 \to \br^2 \) given by \( \mathbf x \mapsto k \mathbf x \) is a dilation at infinity, which contracts at the end of \( \br^2 \). 
Our final theorem says that dilations at infinity are uniform normal generators.

\begin{Thm*}
\label{thm:dilation}
Let \( M \) be a telescoping 2-manifold with a finite number of maximal ends. 
If \( \sigma \in H(M) \) is a dilation at infinity, then \( \sigma \) uniformly normally generates \( H(M) \).
\end{Thm*}

For the case of a finite number of maximal ends, a quantitative version of Theorem~\ref{thm:dilation} is given in Theorem~\ref{thm:dilation2} for a subclass of dilations at infinity; Theorem~\ref{thm:dilation} in this case is then given as a corollary to Theorem~\ref{thm:dilation2}, namely Corollary~\ref{cor:dilation}. 
In the case of a perfect set of maximal ends, a related and stronger result, namely the complete classification of the (uniform) normal generators of \( H(M) \), was established in \cite[Theorem~5.8.1]{VlamisHomeomorphism} (the case of the 2-sphere is due to Anderson \cite{AndersonAlgebraic}, and the case of the 2-sphere with a Cantor set removed is due to Calegari--Chen \cite{CalegariNormal}).

Before continuing, we remark that the theorems above can be extended to several subgroups of the homeomorphism group of a telescoping 2-manifold \( M \) with a nonempty perfect set of maximal ends.
Given a subset \( \boldsymbol\mu \) of the maximal ends of \( M \) of cardinality at most two, Theorem~\ref{thm:distortion}, Theorem~\ref{thm:perfect}, and Theorem~\ref{thm:dilation}  hold with \( \Homeo(M) \) replaced by \( \Homeo(M, \boldsymbol \mu) = \{ h \in \Homeo(M) : h(\boldsymbol \mu) = \boldsymbol\mu\} \) and \( H(M) \) replaced by \( H(M, \boldsymbol \mu) = \{ h \in H(M) : h(\mu) = \mu \text{ for all } \mu \in \boldsymbol \mu\} \). 
This case is not technically mentioned in the proofs below, but all the proofs go through verbatim. 

\subsection*{Motivation}
Building on the language set out by Mann--Rafi, Malestein--Tao introduced the notion of a perfectly self-similar 2-manifold\footnote{Malestein--Tao actually use the term uniformly self-similar, but we will follow the naming convention set out in \cite{VlamisHomeomorphism}.}, where the prototypical example is the 2-sphere with an embedded Cantor space\footnote{A \emph{Cantor space} is any space homeomorphic to \( 2^\bn \), which in turn, is homeomorphic to the standard middle-thirds Cantor set.} removed. 
In the past few years, various boundedness properties---summarized in Theorem~\ref{thm:vlamis} below---have been established for the homeomorphism groups of perfectly self-similar 2-manifolds.
Moreover, the proofs of these various properties mimicked the classical proofs establishing the same properties for homeomorphism groups of spheres. 
This led the author, in a largely expository article \cite{VlamisHomeomorphism}, to expand on these results and make clear the analogy between the more classical results about \( \Homeo(\bS^2) \) and the recent results regarding perfectly self-similar 2-manifolds; in particular, we gave a definition in which the 2-sphere is itself a perfectly self-similar 2-manifold. 

As noted in the beginning of the introduction, the homeomorphism groups of spheres, Euclidean spaces, and annuli share many properties.
The motivation of this article is to find the class of 2-manifolds that fit in an analogy with \( \br^2 \) and \( \bS^1 \times \br \), just as the perfectly self-similar 2-manifolds fit in an analogy with \( \bS^2 \). 
The class of telescoping 2-manifolds with a perfect set of maximal ends is equal to the class of perfectly self-similar 2-manifolds, and hence, their homeomorphism groups behave like that of \( \bS^2 \).
Theorem~\ref{thm:distortion}, Theorem~\ref{thm:perfect}, and Theorem~\ref{thm:dilation}, together with the topological structure of telescoping 2-manifolds uncovered in Section~\ref{sec:topology}, cement the analogy 
between the homeomorphism groups of telescoping 2-manifolds with one, two, and a perfect set of maximal ends and those of \( \br^2 \), \( \bS^1 \times \br \), and \( \bS^2 \), respectively.

The main theorems described above extend the known results for telescoping 2-manifolds with a perfect set of maximal ends (i.e., the perfectly self-similar 2-manifolds) to all telescoping 2-manifolds. 
The results we are extending are summarized in the following theorem. 

\begin{Thm}[{\cite{VlamisHomeomorphism}}]
\label{thm:vlamis}
Let \( M \) be a perfectly self-similar 2-manifold (or equivalently, a telescoping 2-manifold with a perfect set of maximal ends).
    \begin{enumerate}[(1)]
        \item \( \Homeo(M) \) is strongly distorted.
        \item \( H(M) = [\Homeo(M), \Homeo(M)]  \).
        \item \( H(M) \) is uniformly perfect with commutator width at most two. 
        \item If \( h \in H(M) \) acts nontrivially on the maximal ends of \( M \), then \( h \) uniformly normally generates \( H(M) \).
        More precisely, every element of \( H(M) \) is a product of at most eight conjugates of \( h \) and \( h^{-1} \). 
    \end{enumerate}
\end{Thm}

As noted throughout the discussion of our results above, we  prove Theorem~\ref{thm:distortion}, Theorem~\ref{thm:perfect}, and Theorem~\ref{thm:dilation} in the case of a finite set of maximal ends and refer to Theorem~\ref{thm:vlamis} for the case of a perfect set of maximal ends.
We stress that the proofs in the finite maximal end case are distinct from the perfect set of maximal ends case. 

\subsection*{Applications and questions}

Strong distortion is stronger than several common properties in the literature; we quickly recall several such properties here and refer the reader to the introduction of \cite{LeRouxStrong} for a more detailed discussion. A group is \emph{strongly bounded} if every length function on the group is bounded (equivalently, the group has bounded diameter in every left-invariant metric); it has \emph{uncountable cofinality} if it cannot be written as a strictly increasing union of countably many proper subgroups; and it has the \emph{Schreier property} if every countable sequence of elements is contained in a finitely generated subgroup.

It is an exercise to show that a countable group (1) is strongly bounded if and only if it is finite, and (2) has uncountable cofinality if and only if it is finitely generated. 
Hence, strong boundedness and uncountable cofinality are meant to naturally generalize, in a nontrivial way, the notions of finite and finitely generated, respectively, to the setting of uncountable groups. 
As a consequence of Theorem~\ref{thm:distortion}, we have the following.

\begin{Cor}
\label{cor:distortion}
The homeomorphism group of a telescoping 2-manifold is strongly bounded, has uncountable cofinality, and has the Schreier property. 
\qed
\end{Cor}

As the composition of a length function and a homomorphism is again a length function, a strongly bounded group cannot surject onto a group with an unbounded length function. 
It is an exercise to show that every denumerable group admits an unbounded length function, and hence we have the following corollary to Theorem~\ref{thm:distortion}.
 
\begin{Cor}
The homeomorphism group of a telescoping 2-manifold does not have a normal subgroup of denumerable index.
\qed
\end{Cor}

As a consequence of Theorem~\ref{thm:vlamis}(4), the homeomorphism group of a telescoping 2-manifold with a perfect set of maximal ends has finitely many finite-index normal subgroups; in particular, if the underlying 2-manifold is orientable, then the only proper finite-index normal subgroup is the group of orientation-preserving homeomorphisms, and if it is not orientable, then there are no finite-index proper normal subgroups. 

\begin{Question}
Does the homeomorphism group of a telescoping 2-manifold with finitely many maximal ends have finitely many finite-index normal subgroups? If so, can we describe them? 
\end{Question}

We note that Corollary~\ref{cor:distortion} partially strengthens a result of Mann--Rafi. 
A topological group \( G \) is \emph{coarsely bounded} if every continuous left-invariant metric on the group is bounded. 
Note that a group is strongly bounded if and only if it is coarsely bounded when equipped with the discrete topology.
Clearly, a strongly bounded topological group is coarsely bounded. 
The homeomorphism group of a manifold is a topological group when endowed with the compact-open topology. 
Up to a mild hypothesis, Mann--Rafi \cite{MannLarge} classified the infinite-type 2-manifolds whose homeomorphism groups are coarsely bounded; in particular, they are the self-similar 2-manifolds and the telescoping 2-manifolds\footnote{Mann--Rafi's work is in the setting of mapping class groups, but their work is readily adaptable to the setting of homeomorphism groups.}. 
Given Mann--Rafi's classification of coarsely bounded homeomorphism groups for infinite-type surfaces, it is natural to ask the following question.

\begin{Question}
\label{q:sb}
If \( M \) is a 2-manifold such that \( \Homeo(M) \) is strongly bounded, is \( M \) telescoping? 
\end{Question}

As mentioned earlier, Mann--Rafi's classification has an additional hypothesis, which they refer to as tameness, and so it is reasonable to consider the above question with the same hypothesis. 
Note that Mann--Rafi's classification of the coarsely bounded homeomorphism groups does not answer Question~\ref{q:sb}: indeed, by the work of Domat--Dickmann \cite{DomatBig}, there exists a self-similar 2-manifold with a unique maximal end  whose homeomorphism group  is not strongly bounded (on account of having infinite abelianization); in particular, by Corollary~\ref{cor:distortion}, such a 2-manifold cannot be telescoping. 
Therefore, the class of 2-manifolds with strongly bounded homeomorphism groups is a proper subset of those with coarsely bounded homeomorphism group.

As soon as a finite-type surface has enough topology, the mapping class group (defined below) is denumerable, and hence the homeomorphism group admits an unbounded length function. 
So, in the finite-type setting, Question~\ref{q:sb} reduces to several exceptional cases, for which the author is unaware of any results.

\begin{Question}
Are the homeomorphism groups of the projective plane, the Klein bottle, the M\"obius band, the twice-punctured projective plane, and the thrice-punctured sphere coarsely or strongly bounded?
\end{Question}

Theorem~\ref{thm:dilation} shows the existence of normal generators for \( H(M) \), so it is natural to ask if we can classify all normal generators.
This was accomplished in the case of a non-empty perfect set of maximal ends in Theorem~\ref{thm:vlamis}(4); in particular, an element of \( H(M) \) is a normal generator if and only if it induces a nontrivial permutation on the set of maximal ends. 
(The case of empty set of ends corresponds to the 2-sphere, in which case Anderson \cite{AndersonAlgebraic} showed that \( H(M) \) is simple, and hence every nontrivial element is a normal generator.)

\begin{Problem}
Given a telescoping 2-manifold \( M \), classify the normal generators of \( H(M) \). 
\end{Problem}

It follows from Theorem~\ref{thm:vlamis}(4)---but was first observed by Malestein--Tao \cite{MalesteinSelf}---that if \( M \) is a telescoping 2-manifold with a perfect set of maximal ends, then \( H(M) \) is normally generated by an involution.
As a corollary to Theorem~\ref{thm:dilation}, we give an analogous statement in the case of two maximal ends. 

\begin{Cor}
\label{cor:rotation}
Let \( M \) be a telescoping 2-manifold with at least two maximal ends. 
Let \( G = \Homeo^+(M) \) when \( M \) is orientable, and otherwise let \( G = \Homeo(M) \). 
Then there exists an involution \( \theta \in G \) such that every element of \( G \) can be expressed as product of at most 17 conjugates of \( \theta \). 
\end{Cor}

The proof of Corollary~\ref{cor:rotation} is given in Section~\ref{sec:dilations}.
We finish this section with asking if the above corollary holds in the case of a single maximal end.

\begin{Question}
Let \( M \) be a telescoping 2-manifold with a unique maximal end, and let \( G \) be as in Corollary~\ref{cor:rotation}. 
Can \( G \) be generated by involutions? Can \( G \) be generated by torsion?
\end{Question}

\subsection*{Mapping class groups}

The  mapping class group \( \mcg(M) \) of a(n) (orientable) 2-manifold \( M \) is the group of isotopy classes of (orientation-preserving) homeomorphisms \( M \to M \).
In particular, \( \mcg(M) \) is a quotient of a finite-index subgroup of \( \Homeo(M) \). 
For a telescoping 2-manifold \( M \), let \( \Gamma( M ) \) denote the image of \( H(M) \) under this quotient. 

We say an element of \( \Gamma(M) \) is a dilation at infinity if it is has a representative homeomorphism that is a dilation at infinity.
Now, observe that the properties of strong distortion, uniform perfectness, and being a uniform normal generator are inherited by quotients. 
Therefore, Theorem~\ref{thm:distortion}, Theorem~\ref{thm:perfect}, and Theorem~\ref{thm:dilation} can be stated in terms of mapping class groups, which we record here as a single corollary. 

\begin{Cor}
\label{cor:mcg}
If \( M \) is a telescoping 2-manifold, then
\begin{enumerate}[(i)]
\item \( \mcg(M) \) is strongly distorted,
\item \( \Gamma(M) = [\mcg(M), \mcg(M)] \),
\item \( \Gamma(M) \) is uniformly perfect and has commutator width at most two, 
\item each dilation of \( \Gamma(M) \) uniformly generates \( \Gamma(M) \), and
\item there exists an involution \( \theta \in \mcg(M) \) such that each element of \( \mcg(M) \) is a product of at most 17 conjugates of \( \theta \). \qed
\end{enumerate}
\end{Cor}

With Corollary~\ref{cor:mcg}, this article can be viewed as part of the rapidly growing literature on mapping class groups of infinite-type surfaces, known as \emph{big mapping class groups} (see \cite{AramayonaBig} for a survey on big mapping class groups).

\subsection*{Acknowledgements}
The author is grateful to the anonymous referee for their valuable comments and to Kathryn Mann, Fr\'ed\'eric LeRoux, and George Domat for helpful discussions. 
The author acknowledges support from NSF DMS-2212922, PSC-CUNY Award 65331-00 53, and PSC-CUNY Award 66435-00 54. 


\section{Preliminaries}
\label{sec:prelim}

\subsection{Basic notions}
Let \( X \) be a topological space.
Given a subset \( U \) of \( X \), its closure is denoted \( \overline U \); its interior is denoted \( U^\mathrm{o} \); and its \emph{boundary} is the set \( \partial U = \overline U \ssm U^\mathrm o \).

Let \( f \co X \to X \) be a homeomorphism. 
The \emph{support} of \( f \), denoted \( \supp(f) \), is the closure of the set \( \{ x \in X : f(x) \neq x \} \). 
If \( V \subset X \) such that \( \supp(f) \subset V \), then we say that \( f \) is \emph{supported} in \( V \). 

A family of subsets \( \{ V_n\}_{n\in\bn} \) of \( X \) is \emph{locally finite} if, given a compact subset \( K \) of \( X \), \( V_n \cap K \neq \varnothing \) for only finitely many \( n \in \bn \). 
Given homeomorphisms \( f_1, \ldots, f_n \co X \to X \), we let \[ \prod_{j=1}^n f_j = f_1 \circ f_2 \circ \cdots \circ f_n. \]
With this convention, given a sequence of self-homeomorphisms \( \{f_n\}_{n\in\bn} \) of \( X \) such that \( \{ \supp(f_n)\}_{n\in\bn} \) is locally finite, we can define the infinite product \( F =  \prod_{n \in \bn } f_n \) as follows:
given a point \( x \in X \), the maximum of  \( \{ k \in \bn : x \in \supp(f_k) \} \) exists, call it \( m_x \); define \( F(x) = (f_1 \circ f_2 \circ \cdots \circ f_{m_x})(x) \). 
It is readily verified that \( F \co X \to X \) is a homeomorphism. 

We say a subset \( V \) of \( X \) \emph{separates} a point/subset/end \( x \) of \( X \) from another point/subset/end \( y \) of \( X \) if \( x \) and \( y \) are in distinct components of \( X \ssm V \). 
Given two subsets \( Y \) and \( Z \) of \( X \), we say that \( Y \) and \( Z \) are \emph{abstractly homeomorphic} if there is a homeomorphism \( f \co Y \to Z \), and we say they are \emph{ambiently homeomorphic} if there exists a homeomorphism \( h \co X \to X \) such that \( h(Y) = Z \).

\subsection{Surfaces}

All surfaces are assumed to be connected, second countable, and Hausdorff; we use the term 2-manifold to refer to a surface with empty boundary. 
A 2-manifold is of \emph{finite type} if it has finitely generated fundamental group (or equivalently, can be realized as the interior of a compact surface); otherwise, it is of \emph{infinite type}. 

Let \( S \) be a surface. 
We will follow the convention that a \emph{subsurface} of \( S \) refers to a closed subset of \( S \) that is a surface with respect to the subspace topology. 
We let \( \Ends(S) \) denote the space of ends of \( S \), which is a compact, second countable, zero-dimensional, Hausdorff topological space; in particular, it is homeomorphic to a closed subset of \( 2^\bn \). 
If \( U \subset S \) is such that \( \partial U \) is compact, then \( U \) determines a clopen subset of \( \Ends(S) \), which we denote \( \widehat U \).
In fact, every clopen subset of \( \Ends(S) \) is of the form \( \widehat \Sigma \) for some subsurface \( \Sigma \) with connected compact boundary. 
If \( e \) is an end of \( S \) in \( \widehat U \), we say that \( U \) is a \emph{neighborhood} of \( e \) in \( S \) or \( U \) \emph{sees} \( e \). 

An end of \( S \) is \emph{planar} if it admits a neighborhood in \( S \) that can be embedded in \( \br^2 \); otherwise, it is \emph{non-planar}.
An end of \( S \) is \emph{orientable} if it admits an orientable neighborhood in \( S \); otherwise, it is \emph{non-orientable}.
Let \( \Ends_{np}(S) \) denote the set of  non-planar ends of \( S \), and let \( \Ends_{no}(S) \) denote the set of non-orientable ends of \( S \). 
Observe that \( \Ends_{no}(S) \) and \( \Ends_{np}(S) \) are closed subsets of \( \Ends(S) \) and that \( \Ends_{no}(S) \subset \Ends_{np}(S) \).

An essential tool is the classification of surfaces (see Richards \cite{RichardsClassification}):
Two 2-manifolds \( M \) and \( M' \) are homeomorphic if and only if they have the same genus, the same orientability class\footnote{There are four orientability classes: a surface can be orientable or non-orientable of even, odd, or infinite non-orientability.  We forgo the definitions here, but note that a telescoping 2-manifold is either orientable or infinitely non-orientable.}, and \( (\Ends(M), \Ends_{np}(M), \Ends_{no}(M)) \) and \( (\Ends(M'), \Ends_{np}(M'), \Ends_{no}(M')) \) are homeomorphic as triples of spaces. 
(More detail can be found in \cite[Section~3]{VlamisHomeomorphism}). 

Every homeomorphism \( f \co M \to M \) induces a homeomorphism \( \hat f \co \Ends(M) \to \Ends(M) \).  
It follows from Richards's proof of the classification of surfaces that every self-homeomorphism of \( \Ends(M) \) that preserves \( \Ends_{np}(M) \) and \( \Ends_{no}(M) \) is induced by a self-homeomorphism of \( M \).

Following Mann--Rafi \cite{MannLarge}, we define a preorder \( \preceq \) on \( \Ends(S) \) as follows:
Given two ends \( \varepsilon_1, \varepsilon_2 \in \Ends(S) \), we write \( \varepsilon_1 \preceq \varepsilon_2 \) if every open neighborhood of \( \varepsilon_2 \) in \( \Ends(S) \) contains an open set homeomorphic to an open neighborhood of \( \varepsilon_1 \).
Two ends \( \varepsilon_1 \) and \( \varepsilon_2 \) are \emph{comparable} if \( \varepsilon_1 \preceq \varepsilon_2 \) or \( \varepsilon_2 \preceq \varepsilon_1 \); they are are \emph{equivalent} if \( \ve_1 \preceq \ve_2 \) and \( \ve_2 \preceq \ve_1 \). 
An end \( \mu \) is \emph{maximal} if \( \mu \preceq \varepsilon \) implies \( \varepsilon \preceq \mu \) whenever \( \varepsilon \) is an end comparable to \( \mu \). 


\section{Topology of telescoping 2-manifolds}
\label{sec:topology}

We begin this section with giving examples of telescoping 2-manifolds and then continue to develop their topological structure. 

\subsection{Examples}
\label{sec:examples}

As already noted, \( \mathbb S^2 \), \( \mathbb R^2 \), and \( \mathbb S^1 \times \br \) are telescoping, and they serve as our prototypes. 
Since we did not mention it earlier, we point out that every closed disk is a maximal telescope in \( \mathbb S^2 \), the complement of the interior of any closed disk is a maximal telescope in \( \br^2 \), and \( \mathbb S^1 \times [0,\infty) \) is a maximal telescope in \( \mathbb S^1 \times \br \). 
We already discussed that these are the only finite-type telescoping 2-manifolds, so we will now construct infinite-type examples.
Moreover, for each example given below, we given an example of a maximal telescope in the manifold.

We do not rigorously argue that the examples below are in fact telescoping, but simply note this can be checked using a combination of the Jordan--Sch\"onflies theorem, Brouwer's topological characterization of a Cantor space, and the classification of surfaces. 
\begin{enumerate}[(1)]
	\item Let \( C \) be a Cantor space embedded in \( \mathbb S^2 \). 
		Then, \( \mathbb S^2 \ssm C \) is telescoping, and if \( D \) is a disk in \( \mathbb S^2 \) partitioning \( C \) into two clopen sets, then \(  D \ssm C \) is a maximal telescope.
	\item Let \( C_1 \) and \( C_2 \) be Cantor spaces embedded in \( \mathbb S^2 \)  such that \( C_1 \cap C_2 \) is nonempty with at most two points. 
		Let \( A \) be a discrete set of points in \( \mathbb S^2 \ssm (C_1 \cup C_2) \) such that \( \overline A = A \cup C_1 \).  
		Letting \( E = C_1 \cup C_2 \cup A \), we have that \( \mathbb S^2 \ssm E \) is telescoping and, if \( D \) in \( \mathbb S^2 \) is a disk such that  \( D \cap  C_1 \cap C_2 \) is a singleton, then \( D \ssm E \) is a maximal telescope.
	\item Keeping the same notation as in the previous example and taking a connected sum of \( \bS^2 \ssm (C_1 \cup C_2) \) and infinitely many tori (resp., projective planes) using small neighborhoods of the points in \( A \) results in an infinite genus (resp., non-orientable) telescoping 2-manifold.
\end{enumerate}

Another description of Examples (2) and (3) is given in Figure~\ref{fig:telescoping}. 

\begin{figure}\
\centering
\includegraphics{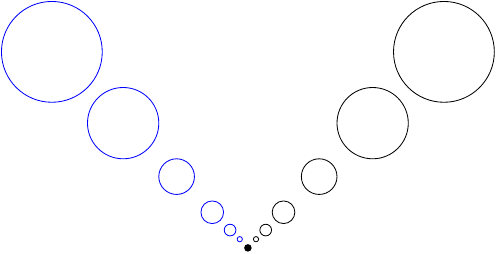}
\caption{To get Example (2), in each blue disk (the sequence of disks on the left) delete a copy of the Cantor set and in each of the black disks (the sequence of disks on the right) delete a copy of a Cantor set and a discrete set of points accumulating onto the Cantor set. To get Example (3), start with Example (2) and then in each black disk use a neighborhood of each point in the deleted discrete set to take a connected sum with a torus.}
\label{fig:telescoping}
\end{figure}

Note that replacing the set \( A \) in Example (2) with a countable set such that whenever the \( \alpha^\text{th} \)-Cantor--Bendixson derivative is non-empty it  accumulates onto \( C_1 \) yields an uncountable family of pairwise non-homeomorphic telescoping 2-manifolds. 
We also note that if \( C_1 \cap C_2 \) is a single point in (2), then the manifold \( \bS^2 \ssm E \) fails to be telescoping with respect to the Mann--Rafi definition of telescoping; hence, the class of Mann--Rafi telescoping 2-manifolds is a proper subclass of the telescoping 2-manifolds as defined here.
In fact, a telescoping 2-manifold---as defined here---is a Mann--Rafi telescoping 2-manifold if and only if it does not have a unique maximal end. 

To finish, we give an example showing that condition (ii) in Definition~\ref{def:telescope} is necessary. 
Consider the orientable two-ended infinite-genus 2-manifold \( L \). 
As \( L \) has only two ends, which are equivalent, they are both maximal, and it therefore makes sense to define \( H(L) \) as we did above, i.e., as the group of orientation-preserving homeomorphisms fixing each end of \( L \). 
Aramayona, Patel, and the author \cite{AramayonaFirst} have shown that there exists an epimorphism \( H(L) \to \bz \), and hence, \( H(L) \) is neither strongly distorted nor uniformly perfect; in particular, \( L \) cannot be telescoping if Theorem~\ref{thm:distortion} and Theorem~\ref{thm:perfect} are to hold.

Now, let \( a \) be a simple closed curve separating the two ends of \( L \), let \( e \) be an end of \( L \), and let \( T \) be the closure of the component of \( L \ssm a \) that sees \( e \). 
If \( S \) is an embedded copy of \( \mathbb D \) in \( T \ssm L \), then it is not difficult to see that \( L \), \( T \), \( e \), and \( S \) satisfy conditions (i) and (iii) of Definition~\ref{def:telescope}. 
However, condition (ii) fails, and hence we see it is a necessary condition for our results.

\subsection{Topological structure}
To motivate this section, let us consider the annulus \( A = \mathbb S^1 \times \br \) and the subset \( P = \mathbb S^1 \times [0,1] \) of \( A \).
If \( \tau \co A \to A \) is given by \( \tau(x,t) = (x, t+1) \), then \( A = \bigcup_{n \in \bz} \tau^n(P) \) and, for every \( m \in \bn \), the subsurface \( P \) is homeomorphic to \( \bigcup_{j=0}^m \tau^j(P) \). 
Equipping \( \br^2 \) with polar coordinates, the above discussion is readily adapted to \( \br^2 \).
The goal of this section is the following: 
\begin{quote}
\emph{Given a telescoping 2-manifold, find a subsurface of the manifold that has the same properties as \( P \) given above.}
\end{quote}
To accomplish this goal, we work through a sequence of lemmas building up an understanding of the topology of telescoping 2-manifolds. 
The first lemma is an exercise in the definitions and shows that the defining end of a maximal telescopes admits a neighborhood basis consisting of maximal telescopes. 

\begin{Lem}
\label{lem:deep_lens}
Let \( M \) be a telescoping 2-manifold, and let \( T \) be a maximal telescope in \( M \).  
Every separating simple closed curve that separates  \( \partial T \) from the defining end of \( T \) bounds a maximal telescope contained in \( T \). 
\qed
\end{Lem}

The following lemma is another exercise in the definitions given thus far, together with the fact that the set of maximal ends of a 2-manifold is closed \cite[Lemma~4.6]{MannLarge}. 

\begin{Lem}
\label{lem:maximal_prop}
\begin{enumerate}[(1)]
\item Every defining end of a maximal telescope in a telescoping 2-manifold is a maximal end.
\item Any two maximal ends of a telescoping 2-manifold are equivalent. 
\item The set of maximal ends of a telescoping 2-manifold is either a singleton, doubleton, or a perfect set. 
\qed
\end{enumerate}
\end{Lem}

Lemma~\ref{lem:maximal_prop}(3) partitions the class of telescoping 2-manifolds into three classes, depending on whether the set of maximal ends is a singleton, doubleton, or a perfect set.
We note that the telescoping 2-manifolds with a unique maximal end are exactly the 2-manifolds that are telescoping by the definition presented here but that are not telescoping in the definition given by Mann--Rafi. 
We also note, as mentioned in the introduction, that the telescoping 2-manifolds with a unique maximal end are a proper subclass of the class of  uniquely self-similar 2-manifolds, and the class of telescoping 2-manifolds with a perfect set of maximal ends is equal to the class the perfectly self-similar 2-manifolds.

A subset  \( K \) of a manifold \( M \) is \emph{displaceable} if there exists a homeomorphism \( f \co M \to M \) such that \( f(K) \cap K = \varnothing \). 
A 2-manifold is \emph{weakly self-similar} if every proper compact subset of the manifold is displaceable.  

\begin{Lem}
\label{lem:weakly_ss}
Every telescoping 2-manifold is weakly self-similar.
\end{Lem}

\begin{proof}
Let \( M \) be a telescoping 2-manifold, and let \( K \) be a compact subset of \( M \).
By Lemma~\ref{lem:deep_lens}, there exists a maximal telescope \( T \) of \( M \) such that \( T \cap K = \varnothing \).
Therefore, there exists a homeomorphism \( f \co M \to M \) such that \( f(M \ssm T) \subset T \); in particular, \( f(K) \cap K = \varnothing \), as \( K \subset M \ssm T \). 
\end{proof}

Lemma~\ref{lem:weakly_ss}, together with \cite[Remark~4.15]{MannLarge}, tells us that every maximal end of a telescoping 2-manifold is \emph{stable}; we will forgo the definition of a stable end, but we refer the reader to \cite[Section~4]{MannLarge} and \cite[Section~4]{VlamisHomeomorphism} for detailed discussions regarding stable ends.
The main point we want to take away from stability is the following lemma, which can be deduced from \cite[Propositions~4.2~\&~4.8]{VlamisHomeomorphism} together with the fact that, whenever there are at least two maximal ends, every maximal telescope partitions the set of maximal ends into two non-empty subsets. 

\begin{Lem}
\label{lem:abs_same_tel}
Let \( M \) be a telescoping 2-manifold.
\begin{enumerate}[(i)]
\item Given two maximal ends of \( M \), there exists an ambient homeomorphism mapping one to the other. 
\item Every maximal end seen by a maximal telescope \( T \) is a defining end for \( T \).
\item Given two defining ends of a maximal telescope \( T \), there exists a homeomorphism supported in \( T \) mapping one to the other. 
\item Any two maximal telescopes of a telescoping 2-manifold are abstractly homeomorphic. 
\item Every extension of a maximal telescope is abstractly homeomorphic to the maximal telescope.
\item If \( M \) has at least two maximal ends and \( T \) is a maximal telescope of \( M \), then \( M \ssm T^\mathrm o \) is a maximal telescope of \( M \). 
\qed
\end{enumerate}
\end{Lem}

Next, we upgrade the abstract homeomorphism in Lemma~\ref{lem:abs_same_tel}(iv) to an ambient homeomorphism. 

\begin{Prop}
\label{prop:amb_same_tel}
Any two maximal telescopes of a telescoping 2-manifold are ambiently homeomorphic. 
\end{Prop}

\begin{proof}
Let \( T_1 \) and \( T_2 \) be maximal telescopes in a telescoping 2-manifold \( M \). 
Let us first assume that \( M \) has at least two maximal ends.
Then, by Lemma~\ref{lem:abs_same_tel}, \( T_1 \) and \( T_2 \) are homeomorphic as are there complements, and therefore, there is an ambient homeomorphism mapping \( T_1 \) onto \( T_2 \). 

We now assume that \( M \) has a unique maximal end. 
Therefore, by Lemma~\ref{lem:maximal_prop}(1), \( T_1 \) and \( T_2 \) have the maximal end as their defining end.
This allows us to choose a separating simple closed curve \( b \) that separates each of \( \partial T_1 \) and \( \partial T_2 \) from the maximal end.
On account of \( T_i \) being a maximal telescope, there exists a homeomorphism \( f_i \co M \to M \) such that \( f_i(b) = T_i \).
To finish, observe that \( (f_2\circ f_1^{-1})(T_1) = T_2 \), as desired.
\end{proof}

The next lemma will allow to introduce the type of surface that will play the analogous role of \( \mathbb S^1 \times [0,1] \) in \( \mathbb S^1 \times \br \). 

\begin{Lem}
\label{lem:tube_exists}
Let \( M \) be a telescoping 2-manifold.
If \( T \) is a maximal telescope in \( M \), then there exists an extension \( S \) of \( T \) such that \( S \) is a maximal telescope. 
\end{Lem}

\begin{proof}
Let \( S' \) be an extension of \( T \).
By Lemma~\ref{lem:abs_same_tel}, there exists a homeomorphism \( f \co S' \to T \). 
Viewing \( T' := f(T) \) as a subset of \( M \), we have that \( T' \) is a maximal telescope and \( T \) is an extension of \( T' \). 
As \( \partial T' \) separates \( \partial T \) from a defining end of \( T \), there exists a homeomorphism \( g \co M \to M \) such that \( g(T') = T \).  
Hence, \( S = g(T) \) is a maximal telescope and it is an extension of \( T \). 
\end{proof}

For the following definition, introducing the notion of a tube, we encourage the reader to revisit Definition~\ref{def:telescope} and recall the technical notion of an extension of a telescope.  

\begin{Def}[Tube]
A subsurface \( P \) of a telescoping 2-manifold is a \emph{tube}\footnote{The term tube is also used in \cite{LanierHomeomorphism}; the definition here is stricter, in the sense that a tube as defined here is also a tube in the sense of \cite{LanierHomeomorphism}, but not vice versa.  It would be more appropriate to use something like predecessor tube here, but we prefer to avoid this unnecessarily cumbersome language.} if there exist maximal telescopes \( T_1 \) and \( T_2 \) such that \( T_2 \) is an extension of \( T_1 \) and \( P = T_2 \ssm T_1^\mathrm o \). 
We say that \( T_1 \) and \( T_2 \) \emph{define} \( P \). 
\end{Def}

Lemma~\ref{lem:tube_exists} guarantees that every telescoping 2-manifold contains a tube. 
Next, we define the subsurface that plays the analogous role of a closed disk in \( \br^2 \). 

\begin{Def}[Capped tube]
A subsurface \( C \) in a telescoping 2-manifold with a unique maximal end is a \emph{capped tube} if there exists a tube \( P \) such that \( C = P \cup U \), where \( U \) is the non-telescope component of \( M \ssm P^\mathrm{o} \). 
\end{Def}

\begin{Prop}
\label{prop:amb_tube}
Any two tubes (resp., capped tubes) in a telescoping 2-manifold are ambiently homeomorphic. 
\end{Prop}

\begin{proof}
Let \( M \) be a telescoping 2-manifold, and let \( P, P' \subset M \) be  tubes in \( M \). 
Let \( T_1 \subset T_2 \) be maximal telescopes defining \( P \), and let \( T_1' \subset T_2' \) be maximal telescopes defining \( P' \). 
By Proposition~\ref{prop:amb_same_tel}, we can assume that \( T_2 = T_2' \). 
Then, \( T_2 \) is an extension of both \( T_1 \) and \( T_1' \). 
By Lemma~\ref{lem:abs_same_tel}(ii)\&(iii),  we may assume that \( T_1 \) and \( T_1' \) share a defining end, call it \( \mu \). 
Therefore, there exists a separating simple closed curve \( b \) in \( T_2 \) that separates \( \mu \) from each of \ \( \partial T_1 \) and \( \partial T_1' \). 
Hence, there exists \( f,f' \co M \to M \) supported in \( T_2 \) such that \( f(\partial T_1) = b \) and \( f'(\partial T_1') = b \), and so \( (f^{-1}\circ f')(T_1') = T_1 \).
In particular, \( (f^{-1}\circ f')(P') = P \), as desired.
The statement for capped tubes follows immediately from the fact that any two tubes are ambiently homeomorphic. 
\end{proof}

The following lemma is the final piece that will allow us to accomplish the goal laid out in the beginning of the section. 

\begin{Lem}
\label{lem:tube_ends}
Let \( M \) be a telescoping 2-manifold, and let \( P \) be a tube in \( M \). 
If \( \mathscr U \) is a clopen subset of \( \Ends(M) \) that is either disjoint from the set of maximal ends of \( M \) or partitions the set of maximal ends into two perfect sets, then \( \widehat P \cup \mathscr U \) is homeomorphic to \( \widehat P \). 
Additionally, if \( M \) has a unique maximal end and \( C \) is a capped tube in \( M \), then \( \widehat C \) is homeomorphic to \( \widehat P \)\footnote{As a consequence, \( C \) is homeomorphic to the surface obtained by capping off one of the boundary components of \( P \) with a disk (and hence the choice of name).}. 
\end{Lem}

\begin{proof}
As \( \widehat P \cup \mathscr U = \widehat P \cup (\mathscr U \ssm \widehat P) \), we may assume that \( \widehat P \cap \mathscr U = \varnothing \). 
Let \( T \) be a component of \( M \ssm P^\mathrm o \) that is a maximal telescope. 
Choose a homeomorphism \( f \co M \to M \) such that \( f(M\ssm T) \subset T \), and by replacing \( \mathscr U \) with \( \hat f (\mathscr U) \), we may assume that \( \mathscr U \subset \widehat T \). 
Now, let \( \mu \) be a defining end of \( T \).
If \( M \) has a finite number of maximal ends, then we have \( \mu \not\in \mathscr U \), as \( \mathscr U \) is forced to satisfy the first condition.
Now, suppose \( M \) has a perfect set of maximal ends. 
Then we can find a maximal telescope \( T' \subset T \) such that \( \mu \notin \widehat T' \).
We can then choose a homeomorphism \( g \co M \to M \) such that \( g(M \ssm T') \subset T' \), in which case, by replacing \( \mathscr U \) with \( \hat g(\mathscr U) \), we can assume that \( \mu \notin \mathscr U \). 

We can then choose a separating closed curve \( b \) that co-bounds a subsurface \( S \) with \( \partial T \) such that \( \widehat S = \mathscr U \). 
Now, there exists a homeomorphism \( \tau \co M \to M \) supported in \( P \cup T \) mapping \( b \) to \( \partial T \); hence, \( \hat\tau(\widehat P \cup \mathscr U) = \widehat P \).
In particular, \( \widehat P \cup \mathscr U \) is homeomorphic to \( \widehat P \), as desired. 

To finish, suppose \( M \) has a unique maximal end and that \( C \) is a capped tube in \( M \).
Then, we can write \( C = U \cup P \), where \( P \) is a tube. 
Now, \( \widehat C = \widehat U \cup \widehat P \), and  so by the above argument, \( \widehat C \) is homeomorphic to \( \widehat P \). 
\end{proof}

We can now give a detailed picture of the topology of telescoping 2-manifolds and  how it mimics that of the plane and the annulus.
Note that a tube in each of the 2-sphere, the plane, and the annulus is homeomorphic to a compact annulus, and also note that a capped tube in the plane is homeomorphic to the closed unit disk. 

We say a subset \( I \) of \( \bz \) is an \emph{interval of integers} if there exists \( j, k \in \bz \cup \{\pm \infty\} \) such that \[ I  = \{ \ell \in \bz : j < \ell < k \}. \]

\begin{Def}[Chain product]
Let \( \Sigma \) be a surface with two boundary components, denoted \( \partial_- \) and \( \partial_+ \), and fix a homeomorphism \( f \co \partial_+ \to \partial_- \). 
Given an interval of integers \( I \), we define the \emph{chain product} of \( \Sigma \) and \( I \), denoted \( \Sigma \tilde\times I \), to be the quotient space \( \Sigma \times I / \sim \), where \( \sim \) is the equivalence relation generated by \( (x,i) \sim (f(x), i+1) \) for all \( x \in \partial_+ \) and for all \( i \in I \ssm \{\sup I\} \). 
\end{Def}

The following proposition is readily deduced from Lemma~\ref{lem:tube_ends} and the classification of surfaces. 

\begin{Thm}
\label{thm:topology}
Let \( M \) be a telescoping 2-manifold, and let \( P \) be a tube in \( M \).
\begin{enumerate}
\item If \( I \) is a finite interval of integers, then \( P \tilde\times I \) is homeomorphic to \( P \). 
\item Every maximal telescope in \( M \) is homeomorphic to \( P \tilde\times \bn \).
\item If \( M \) has at least two maximal ends, then \( M \) is homeomorphic to \( P \tilde\times \bz \).
\item If \( M \) has a unique maximal end, then \( M \) is homeomorphic to the 2-manifold obtained by gluing a disk to the boundary of \( P \tilde\times \bn \).
Moreover, if \( C \) is a capped tube in \( M \) and \( I \) is a finite interval of integers, then \( C \) is homeomorphic to the surface obtained by gluing a disk to one of the boundary components of \( P \tilde\times I \).  
\qed
\end{enumerate}
\end{Thm}

With Theorem~\ref{thm:topology} at hand, it may be more natural to define the notion of a telescoping 2-manifold to be a manifold that admits a subsurface \( P \) that satisfies the conclusions of Theorem~\ref{thm:topology}; such a definition may also be readily generalized to higher dimensions. 
We chose the definition given in the introduction to more closely match the original definition given by Mann--Rafi.
We remark that a definition of telescoping based on Theorem~\ref{thm:topology} provokes a natural comparison to notion of a portable manifold introduced by Burago--Ivanov--Polterovich \cite{BuragoConjugation}. 
The connected component of the identity in the group of diffeomorphisms of a portable manifold exhibits various boundedness properties similar to the homeomorphism group of a telescoping 2-manifold; in particular, every conjugation-invariant norm on the group is bounded and the group is uniformly perfect of width at most two. 
Moreover, as \( \br^2 \) and \( \bS^1 \times \br \) are portable manifolds, the telescoping 2-manifolds are a natural generalization of a portable manifold in dimension two.

Observe that, by Theorem~\ref{thm:topology}(1), a tube in a telescoping 2-manifold, as a surface itself, must either be compact or contain an equivalence class of maximal ends that is infinite.
As an infinite equivalence class of maximal ends is a perfect set, this shows that the end space of an infinite-type telescoping 2-manifold is uncountable, which extends \cite[Proposition~3.6]{MannLarge} to the class of telescoping 2-manifolds as defined in this article. 
Though we will not use this fact, we record it here.

\begin{Cor}
The end space of an infinite-type telescoping 2-manifold is uncountable.
\qed
\end{Cor}

We end with defining a decomposition of a telescoping 2-manifold into tubes that is guaranteed by Theorem~\ref{thm:topology}.

\begin{Def}[Tubular decomposition]
Let \( M \) be a telescoping 2-manifold.
If \( M \) has at least two maximal ends, then a \emph{tubular decomposition} of \( M \) is a sequence of tubes \( \{P_n\}_{n\in\bz} \) with pairwise-disjoint interiors such that \( P_n \) and \( P_{n+1} \) share a boundary component and \( M = \bigcup_{n\in\bz} P_n \). 
Otherwise, \( M \) has a unique maximal end and a \emph{tubular decomposition} of \( M \) is a sequence of tubes \( \{P_n\}_{n\in \bn} \) together with a capped tube \( P_0 \) such that \( P_{n-1} \) and \( P_{n} \) share a boundary component for \( n \in \bn  \)\footnote{We adopt the convention that \( 0 \not\in \bn \); this is relevant to the veracity of several statements in the article.} and such that \( M = \bigcup_{n=0}^\infty P_n \). 
\end{Def}

As a corollary of Theorem~\ref{thm:topology}, we have the following.

\begin{Cor}
Every telescoping 2-manifold admits a tubular decomposition.
Moreover, if \( \{ P_j\}_{j\in J} \) is a tubular decomposition of a telescoping 2-manifold, indexed by an interval of integers \( J \), then given a finite interval of integers \( I \subset J \), the union \( \bigcup_{i\in I} P_i \) is a tube. 
\qed
\end{Cor}

\subsection{An application of Proposition~\ref{prop:amb_tube}}

There are two distinct ways in which we will use Proposition~\ref{prop:amb_tube} below, which we establish here.  
The first is the existence of a choice of the ambient homeomorphism in Proposition~\ref{prop:amb_tube} that is  supported in a tube.
The second will allow us to realize the restriction of an ambient homeomorphism to a tube by a homeomorphism supported in a tube, which is crucial to our factorization results in the following sections.
For these tasks, we will have to restrict to a specific type of  embedding of one tube into another, which is detailed in the following definition. 

\begin{Def}[Nestled tubes]
Given two tubes \( P_1 \) and \( P_2 \) in a telescoping 2-manifold, we say that \( P_1 \) is \emph{nestled} in \( P_2 \), denoted \( P_1 \Subset P_2 \), if \( P_1 \subset P_2 \) and each component of \( P_2 \ssm P_1^\mathrm o \) is a tube. 
We say a closed set \( V \) is \emph{nestled} in a tube \( P \) if \( V \) is contained in a tube that is nestled in \( P \). 
\end{Def}

\begin{Lem}
\label{lem:nestled}
Let \( P \), \( P' \), and \( P'' \) be tubes in a telescoping 2-manifold.
If \( P', P'' \Subset P \), then there exists an ambient homeomorphism \( h \) supported in \( P \) such that \( h(P') = P'' \). 
\end{Lem}

\begin{proof}
By Proposition~\ref{prop:amb_tube}, \( P \) and \( P' \) are homeomorphic as are their complements in \( P \), and so the result follows. 
\end{proof}

Now suppose that \( M \) is a telescoping 2-manifold, \( P \subset M \) is a tube, and \( h \in H(M) \).
Let \( Q \) be a tube such that \( P, h(P) \Subset Q \). 
We would like to find a homeomorphism supported in \( Q \) that agrees with \( h \) on \( P \).
By Lemma~\ref{lem:nestled}, there exists \( g \in H(M) \) supported in \( Q \) such that \( (g\circ h)(P) = P \). 
If \( M \) is orientable, then, by editing \( g \) near \( \partial P \) if necessary, we can assume that \( g \circ h \) restricts to the identity on \( \partial P \), which allows us to write \( g\circ h = g_1 \circ g_2 \) with \( g_1 \) supported in \( P \) and \( g_2 \) supported in \( M \ssm {P}^\mathrm o \). 
It follows that \( h \) and \( g^{-1} \circ g_1 \) agree on \( P \) and \( g^{-1} \circ g_1 \) is supported in \( Q \). 

If \( M \) is not orientable, the issue is that \(  g\circ h \) might reverse the orientation of one or both components of \( \partial P \). 
We can remedy this issue with boundary slides:
the notion of a boundary slide is discussed in \cite[Definition~5.5.2]{VlamisHomeomorphism}. 
Briefly, given a two-holed Klein bottle \( K \) and an oriented separating simple closed curve \( b \) in \( K \) that separates the boundary components of \( K \) and that divides \( K \) into two two-holed projective planes, a boundary slide in \( b \) fixes \( \partial K \) pointwise, fixes \( b \) setwise, and reverses the orientation of \( b \). 
Now, given a component of \( \partial P \), there is an embedding of \( K \) into \( Q \) mapping \( b \) onto the given component of \( \partial P \) and that is disjoint from the other component of \( \partial P \). 
Therefore, we can perform a boundary slide supported in \( Q \) that reverses the orientation of a given component of \( \partial P \) while fixing the other component pointwise. 
Post-composing \( g \) with one or two boundary slides appropriately, we can assume that \( g \circ h \) fixes the orientation of each component of \( \partial P \), allowing us to proceed as in the orientable case.
This discussion establishes the following proposition.

\begin{Prop}
\label{prop:local_approx}
Let \( M \) be a telescoping 2-manifold,  let \( P \subset M \) be a tube, and let \( h \in H(M) \). 
If \( Q \) is a tube such that \( P, h(P) \Subset Q \), then there exists \( g \in H(M) \) supported in \( Q \) such that \( g \circ h \) restricts to the identity on \( P \).
\qed 
\end{Prop}


\section{Fragmentation and uniform perfectness}

On account of Theorem~\ref{thm:vlamis}, for the remainder of the article we restrict ourselves to considering telescoping 2-manifolds with a finite number of maximal ends.

\begin{Def}[Standard union of tubes]
A subset \( A \) of a telescoping 2-manifold is called a \emph{standard union of tubes} if there exist an infinite interval of integers \( J \) and a locally finite family \( \{A_j\}_{j\in J} \) consisting of pairwise-disjoint tubes  such that 
    \begin{itemize}
        \item \( A = \bigcup_{j\in J} A_j \),
        \item \( A_{j+1} \) separates \( A_{j} \) from \( A_{j+2} \) for all \( j \in J \), and 
        \item each component of \( M \ssm A^\mathrm o \) is either a tube or capped tube. 
    \end{itemize}
\end{Def}

We will generally have that \( J = \bn\cup\{0\} \) or \( J = \bz \) depending on whether the 2-manifold has one or two maximal ends, respectively; in particular, we will generally assume \( 0 \in J \). 

\begin{Lem}
\label{lem:fragmentation2}
Let \( M \) be a telescoping 2-manifold.
Let \( \{ A_j \}_{j\in J} \) and \( \{ B_j \}_{j \in J} \) be a locally finite families of pairwise-disjoint tubes in \( M \) such that 
    \begin{itemize}
        \item \( A = \bigcup A_j \) and \( B = \bigcup B_j \) are each a standard union of tubes, 
        \item \( A_j \cap B_k \neq \varnothing \) if and only if \( j \in \{k, k+1\} \), and
        \item each component of \( A \cap B \), \( A \ssm B^\mathrm o \), and \( B \ssm A^\mathrm o \) is a tube.
    \end{itemize}
If \( \alpha, \beta \in H(M) \) such that each component of \( \supp(\alpha) \) and \( \supp(\beta) \) is nestled in a component \( A \) and \( B \), respectively, and such that both \( \alpha \) and \( \beta \) restrict to the identity on \( A_0 \cup B_0 \), then there exists \( \alpha', \beta' \in H(M) \) such that \( \alpha\circ \beta = \alpha'\circ \beta' \) and \( \alpha' \) and \( \beta' \) are each commutators of elements supported in \( A \) and \( B \), respectively.
\end{Lem}

The proof of Lemma~\ref{lem:fragmentation2} is motivated by an argument of Le Roux--Mann \cite{LeRouxPrivate} correcting an oversight in the proof of \cite[Theorem~4.1]{LeRouxStrong}.

\begin{proof}[Proof of Lemma~\ref{lem:fragmentation2}]
The key observation, used repeatedly in the proof, is as follows: given an element \( h \in H(M) \) whose support is nestled in a tube \( P \)  and given a tube \( P' \) nestled in \( P \),  there exists a homeomorphism \( \sigma \) supported in \( P \) such that \( \sigma \circ h \circ \sigma^{-1} \) is supported in \( P' \). 
Note that this observation is an immediate consequence of Lemma~\ref{lem:nestled}. 

For \( j \in J \), let \( G_j \) and \( H_j \) be the subgroups of \( H(M) \) consisting of homeomorphisms of \( M \) supported in \( A_j \) and \( B_j \), respectively. 
Applying the observation above, we have the following: given \( g \in G_j \) (resp., \( h \in H_j \)) with support contained in a tube nestled in \( A_j \) (resp., \( B_j \)), there exists a conjugate of \( g \) (resp., \( h \)) with support nestled in \( A_j \cap B_j \) (resp., \( B_j \cap A_{j+1}) \).

We now proceed recursively editing \( \alpha \) and \( \beta \) to obtain the desired \( \alpha' \) and \( \beta' \). 
Let \( T \) and \( T' \) denote the closures of the complementary components of \( M \ssm (A_0 \cup B_0) \). 
At least one of \( T \) and \( T' \) is a maximal telescope; we assume that \( T \) is a maximal telescope, and we will recursively edit \( \alpha \) and \( \beta \) in \( T \). 
If \( T' \) is also a maximal telescope, then the same process can be followed independently in \( T' \); otherwise, by the hypotheses, both \( \alpha \) and \( \beta \) restrict to the identity on \( T' \) as it is disjoint from \( A \cup B \). 

Let \( \alpha_j \) (resp., \( \beta_j \)) be the homeomorphism that agrees with \( \alpha \) (resp., \( \beta \)) on \( A_j \) (resp., \( B_j \)) and that restricts to the identity elsewhere. 
Let \( \gamma_1 \) be a conjugate of \( \alpha_1 \) with support nestled in \( A_1 \cap B_1 \), and set \[ \alpha_1' = \alpha_1 \circ \gamma_1^{-1}. \]
Note that since \( \gamma_1 \) is a conjugate of \( \alpha_1 \) in \( G_1 \), we have that \( \alpha_1' \) is a commutator of elements in \( G_1 \). 
Now, \( \gamma_1\circ \beta_1 \) has support nestled in \( B_1 \), so we may choose a conjugate \( \delta_1 \) of \( \gamma_1 \circ \beta_1 \) whose support is nestled in \( B_1 \cap A_2 \); set \[ \beta_1' = 
\delta_1^{-1}\circ\gamma_1\circ \beta_1. \]
Note that \( \beta_1' \) is a commutator of elements in \( H_1 \).

For the next step, let \( \gamma_2 \) be a conjugate of \( \alpha_2 \circ \delta_1 \) whose support is nestled in \( A_2 \cap B_2 \), and set \[ \alpha_2' = \alpha_2 \circ \delta_1\circ  \gamma_2^{-1}, \] which is a commutator in \( G_2 \). 
Let \( \delta_2 \) be a conjugate of \( \gamma_2 \circ \beta_2 \) whose support is nestled in \( B_2 \cap A_3 \), and set \[ \beta_2' = \delta_2^{-1}\circ \gamma_2 \circ \beta_2,  \] which is a commutator in \( H_2 \). 
Continuing in this fashion, we obtain \( \gamma_j \) and \( \delta_j \)  with supports nestled in \( A_j \cap B_j \) and \( B_j \cap A_{j+1} \), respectively, and such that  \( \alpha_j' := \alpha_j \circ \delta_{j-1} \circ \gamma_j^{-1} \) and \( \beta_j' := \delta_j^{-1} \circ \gamma_j \circ \beta_j \) are commutators in \( G_j \) and \( H_j \), respectively, for each \( j \in \bn \). 

Set  \( \alpha' = \prod \alpha_j' \) and \( \beta' = \prod \beta_j' \).
As we can write \( \alpha_j' = [f_j, g_j] \) with \( f_j \) and \( g_j \) supported in \( A_j \), we can write
\begin{align*}
	\alpha' 	&= \prod \alpha_j' \\
		 	&= \prod [f_j, g_j] \\
			&= \left[\prod f_j, \prod g_j \right]
\end{align*}
where the last equality uses that the disjointness of the \( A_j \). 
Therefore, \( \alpha' \) is a commutator of elements supported in \( A \).
A similar argument shows that \( \beta' \) is a commutator of elements supported in \( B \). 

To finish, we claim \( \alpha\circ\beta|_T = \alpha'\circ\beta'|_T \). 
Let us show this equality pointwise. 
First observe that if \( x\in T \) is not an element of \( A_j \cup B_j \) for \( j > 0 \), then \( \alpha'\beta'(x) = x=\alpha\beta(x) \), so we may assume that \( x \in A_j \cup B_j  \) for some \( j > 0 \).
First assume that \( x \in B_j \).
Then, as \( \beta(x) = \beta_j(x) \), one of the following holds: \( \beta(x) \in A_j \), \( \beta(x) \in A_{j+1} \), or \( \beta(x) \in B_j \ssm A \). 
In the first case, we have 
\begin{align*}
(\alpha'\circ\beta')(x)	&= (\alpha_j'\circ \beta_j')(x) \\
				&= ([\alpha_j \circ \delta_{j-1}\circ \gamma_j^{-1}]\circ[\delta_j^{-1}\circ \gamma_j \circ \beta_j])(x) \\	
				&= (\alpha_j \circ \delta_{j-1}\circ\delta_j^{-1}\circ \beta_j)(x)\\
				&= (\alpha_j\circ \beta_j)(x)\\
				&= (\alpha\circ\beta)(x),
\end{align*}
where the third equality uses the fact that \( \gamma_j \) commutes with \( \delta_j \) (on account of them having disjoint supports),  the fourth equality uses the fact that \( \beta_j(x) \in A_j \cap B_j \) (and hence is in neither the support of \( \delta_j \) nor \( \delta_{j-1} \)), and the final equality uses the assumption that \( \beta(x) \in A_j \).
A similar argument handles the second case, and the third follows from the fact that \[ (\alpha'\circ\beta')|_{B \ssm A} = \beta'|_{B\ssm A} = \beta|_{B \ssm A} = (\alpha\circ\beta)|_{B\ssm A} .\]
Therefore, \( \alpha\circ\beta \) and \( \alpha'\circ\beta' \) agree on \( B \). 
The last possibility is that \( x\in A \ssm B \).
In this case, we have 
\[ (\alpha'\circ\beta')|_{A \ssm B} = \alpha'|_{A\ssm B} = \alpha|_{A \ssm B} = (\alpha\circ\beta)|_{A\ssm B} .\]
Therefore, \( \alpha'\circ\beta'|_T = \alpha\circ\beta|_T \), as desired. 
As already noted, if \( T' \) is also a maximal telescope, we repeat the above steps, independently, in \( T' \) to finish the proof. 
\end{proof}

The proof of the following proposition uses the proof of \cite[Lemma~2.2]{LeRouxStrong} as an outline.

\begin{Prop}
\label{prop:fragmentation}
    Let \( M \) be a telescoping 2-manifold with a finite number of maximal ends.
    Given a sequence \( \{h_n\}_{n\in\bn} \) in \( H(M) \), there exist standard unions of tubes  \( A \) and \( B \) in \( M \) and sequences \( \{\alpha_n\}_{n\in\bn} \), \( \{ \beta_n \}_{n\in\bn} \), and \( \{ \gamma_n \}_{n\in\bn} \) in \( H(M) \) such that \( \supp(\alpha_n) \subset A \), \( \supp(\beta_n) \subset B \), \( \gamma_n \) is supported in a tube, and \(  h_n = \alpha_n \circ \beta_n \circ \gamma_n \).
    Moreover, 
        \begin{itemize}
            \item each component of \( A \cap B \), \( A \ssm B^\mathrm o \), and \( B \ssm A^\mathrm o \) is a tube,
            \item each component of \( \supp(\alpha_n) \) is nestled in a component of \( A \),
            \item each component of \( \supp(\beta_n) \) is nestled in a component of \( B \), 
            \item \( \alpha_n \) and \( \beta_n \) restrict to the identity on \( A_0 \cup B_0 \),
            \item the support of \( \gamma_n \) is nestled in  a component of \( M \ssm B^\mathrm o \), and
            \item \( \alpha_n \) and \( \beta_n \) can each be expressed as a  commutator of homeomorphisms supported in \( A \) and \( B \), respectively.
        \end{itemize}
      	Lastly, there exists locally finite families \( \{ A_j \}_{j\in J} \) and \( \{ B_j \}_{j\in J} \), each consisting of pairwise disjoint tubes,  such that \( A = \bigcup A_j  \), \( B = \bigcup B_j  \), and \( B_j \cap A_k \neq \varnothing \) if and only if \( k \in \{j, j+1\} \).
\end{Prop}

\begin{proof}
Let us assume that \( M \) has two maximal ends; the other case is similar. 
Let \( \{P_j\}_{j\in\bz} \) be a tubular decomposition of \( M \). 
For \( k \in \bz \), let \( I_k = \bigcup_{j = -k}^k P_j \), let \( I_k^+ \) denote the component \( M \ssm I_k \) intersecting \( P_{k+1} \) nontrivially, and similarly, let \( I_k^- \) denote the component of \( M \ssm I_k \) intersecting \( P_{-k-1} \) nontrivially.

We will construct \( A \) recursively.
To do so, we recursively construct a sequence of three natural numbers \( i_n \), \( j_n \), and \( k_n \). 
Set \( j_0 = 0 \). 
For \( n \in \bn \), recursively define \( i_n \), \( j_n \), and \( k_n \) as follows:
Set \( i_n = j_{n-1}+2 \).
Next, define \( k_n \in \bn \) by
    \[
    	k_n=\min\left\{ k \in \bn :  h_\ell^{-1}(I_{i_n}) \subset I_k \text{ for } \ell\leq n \right\}+1,
    \]
so as to guarantee that  \( h_\ell(P_{k_n}) \subset I_{i_n}^+ \) and \( h_\ell(P_{-k_n}) \subset I_{i_n}^- \) for each \( \ell \leq n \). 
Finally, define \( j_n \) by
    \[
    	j_n = \min\left\{ j \in \bn : h_\ell(I_{k_n}) \subset I_j \text{ for } \ell\leq n\right\}+1,
    \]
so as to guarantee that \( h_\ell(P_{k_n}) \subset \bigcup_{\ell = i_n}^{j_n} P_\ell \) and \( h_\ell(P_{-k_n}) \subset \bigcup_{\ell = -j_n}^{-i_n} P_\ell \) for \( \ell \leq n \) (see Figure~\ref{fig:factor}).

\begin{figure}[t]
\centering
\includegraphics{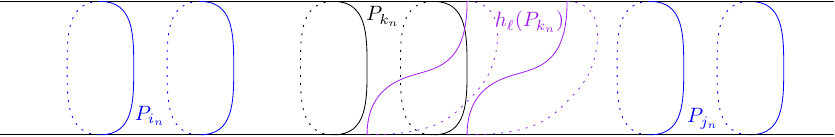}
\caption{A schematic of the choice of \( i_n \), \( j_n \), and \( k_n \).}
\label{fig:factor}
\end{figure}

Now, we define \( A \).
Let \( A_0 \) be a tube satisfying \( P_0 \Subset A_0 \Subset   I_{k_1-1} \), and for \( n \in \bn \), set 
    \[
        A_n =  \bigcup_{\ell = i_n}^{j_n} P_\ell \, \text{  and  } \, A_{-n} = \bigcup_{\ell = -j_n}^{-i_n} P_\ell,
    \]
so \( A_n \) and \( A_{-n} \) are tubes. 
Then,  for \( n \in \bn \), we have that \( P_{\pm k_n}, h_\ell(P_{\pm k_n}) \Subset A_{\pm n} \) for all \( \ell \leq n \).
Now, set \( A = \bigcup_{n\in\bz} A_n \), so that \( A \) is a standard union of tubes. 

Next, fix \( \ell \in \bn \). 
Given the setup above, for \( n \geq \ell \), we can apply Proposition~\ref{prop:local_approx}, to find an ambient homeomorphism of \( M \) supported in \( A_{-n} \cup A_n  \) that agrees with \( h_\ell \) on \( P_{-k_n} \cup P_{k_n}  \). 
It follows that there exists an ambient homeomorphism \( \alpha_\ell \) of \( M \) \ such that \( \alpha_\ell \) agrees with \( h_\ell \) on \( \bigcup_{n \geq \ell } (P_{-k_n} \cup P_{k_n}) \) and is supported in \( \bigcup_{n\geq \ell} (A_n \cup A_{-n} ) \subset A \). 

Now, \( \alpha_\ell^{-1} \circ h_\ell \) restricts to the identity on \( P_{\pm k_m} \) for \( m \geq \ell \), and therefore, \( \alpha_\ell^{-1} \circ h_\ell \) can be factored as \( \beta_\ell \circ \gamma_\ell \), where \( \gamma_\ell \) is supported in the tube \( I_{k_\ell-1} \) and \( \beta_\ell \) is supported in the complement of \( I_{k_\ell}^\mathrm o \). 
To finish, set \( B \) to be the closure of \( M \ssm \bigcup_{n\in\bn} (P_{-k_n} \cup P_{k_n}) \) and, for \( n \in \bz \), let \( B_n \) denote the component of \( B \)  sharing boundary with \( P_{k_{n}} \) and \( P_{k_{n+1}} \), where we let \( k_0 = 0 \).
We then have that \( B \) is a standard union of tubes, \( \beta_\ell \) is supported in \( \bigcup_{n\geq \ell} (B_n\cup B_{-(n+1)}) \subset B \), and \( A_m \cap B_n \neq \varnothing \) if and only if \( n \in \{m, m+1\} \).  
Hence, the factorization \( h_\ell = \alpha_\ell \circ \beta_\ell \circ \gamma_\ell \) is as desired. 
To finish, we may apply Lemma~\ref{lem:fragmentation2} to edit \( \alpha_\ell \) and \( \beta_\ell \) so that they can be realized as commutators of elements supported in \( A \) and \( B \), respectively. 
\end{proof}

We now turn to establishing the uniform perfectness of \( H(M) \). 
The full strength of Proposition~\ref{prop:fragmentation} is not necessary for this task; therefore, as a corollary, we give a weaker version that suits the purpose at hand. 
In the notation of the proof of Proposition~\ref{prop:fragmentation}, the support of \( \gamma_1 \) is nestled in \( I_{k_1} \), and so replacing \( A \) with \( A \cup I_{k_1} \), replacing \( \alpha_1 \) with \( \alpha_1 \circ \gamma_1 \), and ignoring \( h_n \) for \( n > 1 \), we have the following corollary. 

\begin{Cor}
\label{cor:simple_fragmentation}
Let \( M \) be a telescoping 2-manifold with a finite number of maximal ends.
Every element of \( H(M) \) can be factored as a product of two homeomorphisms each of which is supported in a standard union of tubes. 
\qed
\end{Cor}

\begin{Lem}
\label{lem:commutator}
Let \( M \) be a telescoping 2-manifold with a finite number of maximal ends. 
Any element of \( H(M) \) supported in a standard union of tubes can be expressed as a commutator of two elements in \( H(M) \). 
\end{Lem}

\begin{proof}
Let \( h \in H(M) \) be supported in a standard union of tubes, call it \( A \), and let \( \{A_j\}_{j\in J} \) be a locally finite family of pairwise-disjoint tubes such that \( A = \bigcup A_j \). 
Pick a tube \( P \) nestled in a component of \( M \ssm A^\mathrm o \) such that there is a component, call it \( T_+ \), of \( M \ssm P^\mathrm o \) containing \( A_j \) for all \( j \geq 0 \). 
Note \( T_+ \) is a maximal telescope. 
Let \( T_- \) denote the other component of \( M \ssm P^\mathrm o \).
We can factor \( h = h_+ \circ h_- \) with \( h_\pm \) supported in \( T_\pm \). 

Let us work with \( h_+ \).
Using Theorem~\ref{thm:topology}, we see there exists a homeomorphism \( \sigma_+ \in H(M) \) supported in \( T_+ \) such that \( \sigma_+(A_j) = A_{j+1} \) for all \( j \geq 0 \). 
Observe that \( \{\supp(\sigma_+^n \circ h_+ \circ \sigma_+^{-n})\}_{n \in \bn} \) is locally finite, and hence, we can define \( \vp_+ := \prod_{n=0}^\infty (\sigma_+^n \circ h_+ \circ \sigma_+^{-n}) \).
It is a standard exercise to show that \( h_+ = [\vp_+, \sigma_+] \). 

If \( M \) has a unique maximal end, then \( h_- \) is the identity, and we are done.
Otherwise, we can similarly define \( \vp_- \) and \( \sigma_- \) such that \( h_- = [\vp_-, \sigma_-] \).
Set \( \vp = \vp_+\circ \vp_- \) and \( \sigma = \sigma_+ \circ \sigma_- \). 
Now, \( \vp_\pm \) and \( \sigma_\pm \) are supported in \( T_\pm \), and as \( T_+ \cap T_- = \varnothing \), we have that \( h = [\vp, \sigma] \). 
\end{proof}

As an immediate consequence of Corollary~\ref{cor:simple_fragmentation} and Lemma~\ref{lem:commutator}, we have:

\begin{Thm}
\label{thm:uniform2}
Let \( M \) be a telescoping 2-manifold with a finite number of maximal ends. 
Then, \( H(M) \) has commutator width at most two. 
\qed
\end{Thm}


\section{Strong distortion}

We take the proof for strong distortion of \( \Homeo(\br^n) \) given by Le Roux--Mann \cite{LeRouxStrong} as a guide for establishing strong distortion for telescoping 2-manifolds.  
Le Roux--Mann's arguments rely on the structure of the real line and use polar coordinates for \( \br^n \); however, the results from Section~\ref{sec:topology}, establishing the topological structure of telescoping 2-manifolds, provide us with the necessary setup to execute the general strategy laid out by Le Roux--Mann. 

Given a sequence of homeomorphisms, Proposition~\ref{prop:fragmentation} gives us a decomposition of each homeomorphism in the sequence into a product of elements with corresponding factors having common support. 
The next lemma will let us displace these supports all at once so that we can realize all the elements (or rather conjugates of the elements) of the sequence acting on the manifold simultaneously. 
This will allow us to apply a known commutator trick to finish the argument. 

\begin{Lem}
\label{lem:translations}
Let \( M \) be a telescoping 2-manifold, let \( T \) be a maximal telescope in \( M \), and let \( \{ P_j \}_{j\in\bn} \) be a locally finite family of pairwise-disjoint tubes contained in \( T^\mathrm o \).
If the closure of each component of \( T \ssm \bigcup P_j \) is a tube and  \( P_{j+1} \) separates \( P_j \) from \( P_{j+2} \), then there exists \( \sigma, \tau \in H(M) \)  supported in \( T \)  such that 
\begin{enumerate}[(i)]
\item \( \{ \sigma^n(P_j), \tau^m(P_i) : n,m \geq 0, i,j \in \bn \} \) is a locally finite family of pairwise disjoint tubes, and
\item \( \sigma^n(P_j) \cap \tau^m(P_i) = \varnothing \) for all \( n,m \in \bz \) and for all \( i,j \in \bn \) unless \( n = m = 0 \) and \( i = j \). 
\end{enumerate}
\end{Lem}

\begin{proof}
Observe that, using Theorem~\ref{thm:topology}, if \( \{P_j'\}_{j\in\bn} \) is another locally finite family of pairwise disjoint tubes in \( T^\mathrm o \) such that the closure of each component of \( T \ssm \bigcup P'_j \) is a tube and such that \( P'_{j+1} \) separates \( P'_j \) from \( P'_{j+2} \), then there exists a homeomorphism \( f \) supported in \( T \) such that \( f(P'_j) = P_j \) for each \( j \in \bn \). 
Therefore, to prove the statement, it enough to show that some such family of tubes exists; in particular, we will construct the family \( \{P_j\}_{j\in\bn} \) with the desired properties, rather than work with a provided family.

Choose a nested sequence \( \{ T_n \}_{n\in\bn} \) of maximal telescopes contained in \( T \) such that 
    \begin{itemize}
        \item \( T \ssm T_1^\mathrm o \) is a tube in \( M \),
        \item \( T_n \ssm T_{n+1}^\mathrm o \) is a tube \( Q_n \) for all \( n \in \bn \), and
        \item \( \bigcap_{n\in\bn} T_n  = \varnothing \). 
    \end{itemize}
Let \( \sigma \in \Homeo(M) \) be supported in \( T_2 \) such that \( \sigma(T_{2+i}) = T_{3+i} \) for all \( i \in \bn \); in particular, \( \sigma(Q_1) = Q_1 \), \( \sigma(Q_2) = Q_2 \cup Q_3 \), and \( \sigma(Q_{2+i}) = Q_{3+i} \) for all \( i \in \bn \). 
The existence of \( \sigma \) is a consequence of Theorem~\ref{thm:topology}(i--ii) and Proposition~\ref{prop:amb_tube}.

We will now construct \( \tau \). 
Let \( R \) be a tube nestled in \( \sigma^{-1}(Q_3) \subset Q_2 \).
Observe that \( \{ \sigma^n(R) \}_{n\in\bn} \) is a locally finite family of tubes and \( \sigma^n(R) \cap \sigma^m(R) = \varnothing \) whenever \( n \) and \( m \) are distinct integers. 
Next, choose a sequence \( \{R_k\}_{k\in\bn} \) of pairwise-disjoint tubes nestled in \( R \); such a sequence is guaranteed to exist by Theorem~\ref{thm:topology}(i). 
For each \( k \in \bn \), let \( P_k \) and \( P_k' \) be disjoint tubes nestled in \( \sigma^k(R_k) \).
Then, by Lemma~\ref{lem:nestled}, we can choose a homeomorphism \( h_k \) supported in \( \sigma^k(R_k) \) such that \( h_k(P_k) = P_k' \). 
As \( \{ \sigma^k(R_k)\}_{k\in\bn} \) is a locally finite family consisting of pairwise-disjoint sets and the support of \( h_k \) is contained in \( \sigma^k(R_k) \), we can  define \( h = \prod_{k\in\bn} h_k \). 
Now, choose \( \tau_0 \) such that:
\begin{enumerate}[(i)]
\item \( \tau_0 \) is supported in \( T_1 \), so that \( \tau_0(Q_1) = Q_1 \cup Q_2 \),
\item \( \tau_0(Q_2) = Q_3 \), and
\item \( \tau_0 \) agrees with \( \sigma \) on \( T_3 \).
\end{enumerate}
Observe that \( \tau_0^{-n}(R) \subset Q_1 \) for all \( n \in \bn \) while \( \sigma^{-m}(R) \subset Q_2 \) for all \( m \in \bn \); in particular, \( \tau_0^{-n}(R) \cap \sigma^{-m}(R) = \varnothing \) for all \( n, m \in \bn \). 
Set \( \tau = \tau_0 \circ h \). 

Now, for \( n \in \bn \), observe that \( T \ssm T_n \) contains \( \sigma^m(P_k) \) (resp., \( \tau^m(P_k) \)) if and only if \( m+k \leq n - 3 \).
Hence, as every compact subset of \( T \) is contained in \( T \ssm T_n \) for some \( n \),  (i) holds, that is, \( \{ \sigma^n(P_j), \tau^m(P_i) : n,m \geq 0, i,j \in \bn \} \) is a locally finite family of subsets in \( T \).
Finally, it is readily verified that (ii) holds, that is, \( \sigma^n(P_j) \cap \tau^m(P_i) = \varnothing \) for all \( n,m \in \bz \) and for all \( i,j \in \bn \) unless \( n = m = 0 \) and \( i = j \). 
\end{proof}

We can now establish a quantitative version of strong distortion for \( H(M) \). 

\begin{Thm}
\label{thm:distortion2}
Let \( M \) be a telescoping 2-manifold with a finite number of maximal ends.
Given a sequence \( \{h_n \}_{n\in\bn} \) in \( H(M) \), there exists a subset \( \mathcal S \) of \( H(M) \) of cardinality 9 such that \( h_n \) can be expressed as a word in \( \mathcal S \) of length \( 26n+12 \). 
\end{Thm}

\begin{proof}
Apply Proposition~\ref{prop:fragmentation} to the sequence \( \{ h_n \} \): let \( A = \bigcup A_j \) and \( B = \bigcup B_j \) be the resulting standard unions of tubes, and let \( h_n = \alpha_n \circ \beta_n \circ \gamma_n \) be the resulting fragmentation.

Let \( T^+ \) be the maximal telescope sharing a boundary component with \( A_0 \) and containing \( B_0 \).
If \( M \) has two maximal ends, then we let \( T^- \) denote the maximal telescope sharing a boundary component with \( B_0 \) and containing \( A_0 \). 
By assumption, \( \alpha_n \) and \( \beta_n \) each restrict to the identity on \( A_0 \cup B_0 \), and hence, we can write \( \alpha_n = \alpha^+_n \circ \alpha^-_n \) and \( \beta_n = \beta^+_n \circ \beta^-_n \) with \( \supp(\alpha_n^+), \supp(\beta_n^+) \subset T^+ \) and with \( \supp(\alpha_n^-), \supp(\beta_n^-) \subset T^- \) if \( M \) has two maximal ends or with \( \alpha_n^- \) and \( \beta_n^- \) equal to the identity if \( M \) has a unique maximal end.

Let us first work with the \( \alpha_n^+ \). 
Let \( \sigma_+ \) and \( \tau_+ \) be the homeomorphisms supported in \( T^+ \) obtained by applying Lemma~\ref{lem:translations} to \( T^+ \) and the tubes \( \{A_j\}_{j>0} \). 
Write \( \alpha^+_n = [f^+_n, g^+_n] \) with the supports of \( f_n \) and \( g_n \) contained in \( A \). 
Lemma~\ref{lem:translations}(ii) guarantees the existence of the following infinite products: 
	\[
		F^+ := \prod_{n=0}^\infty \sigma_+^n f^+_n \sigma_+^{-n} \, \,  \text{ and } \, \, 
		G^+ := \prod_{n=0}^\infty \tau_+^n g^+_n \tau_+^{-n}
	\]
If \( M \) has two maximal ends, we define \( \sigma_- \), \( \tau_- \), \( F^- \) and \( G^- \) analogously; otherwise, set \( \sigma_- \), \( \tau_- \),  \( F^- \) and \( G^- \) to be the identity.
Define \( \sigma = \sigma^+ \circ \sigma^- \), \( \tau = \tau^+\circ \tau^- \), \( F = F^+\circ F^- \) and \( G = G^+ \circ G^- \).

Now observe that for each \( n \in \bn \), by Lemma~\ref{lem:translations}(ii), the intersection of the support of \( \sigma^{-n} \circ F \circ \sigma^n \) and the support of \( \tau^{-n} \circ G \circ \tau^n \) is \( A \), and hence their commutator vanishes on the complement of \( A \).
Therefore, for each \( n \in \bn \), 
	\[
		\alpha_n = [\sigma^{-n} \circ F \circ \sigma^n,\tau^{-n}\circ G \circ \tau^n] 
	\]
It follows that \( \alpha_n \) can be written as a word in \( \{ F, G, \sigma, \tau \} \) of length \( 8n +4 \).
Repeating the above process for the \( \beta_n \), we see that we can write \( \alpha_n \circ \beta_n \) in a word of length \( 16n+8 \) in an alphabet of size 8. 

It is left to consider the \( \gamma_n \). 
If \( M \) has two (resp., one) maximal ends, let \( \{ P_n \}_{n\in\bn} \) be a sequence of tubes (resp., capped tubes) such that \( P_n \subset P_{n+1} \), such that \( M = \bigcup P_n \), such that \( P_1 \subset A_1 \),  such that \( P_{n+1} \ssm P_n^\mathrm o \) is a tube, and such that \( \supp(\gamma_n) \subset P_n \). 
Choose  \( \vp \in H(M) \) such that \( \vp^n(P_n) \subset P_1 \); the existence of such a homeomorphism can be deduced from Proposition~\ref{prop:amb_tube} and Theorem~\ref{thm:topology}.
Then, \( \{ \vp^n \circ \gamma_n \circ \vp^{-n} \}_{n\in\bn} \) is a sequence of homeomorphisms each of which is supported in \( P_1 \subset A \). 
Using our above argument, this means that \( \vp^n \circ \gamma_n \circ \vp^{-n} \) can be expressed as a word in \( \{F, G, \sigma, \tau \} \) of length \( 8n+4 \), and hence, \( \gamma_n \) can be expressed as a word in \( \{F, G, \sigma, \tau, \vp \} \) of length \( 10n + 4 \).

All together, there exists a set \( \mathcal S \) of cardinality 9 such that, for each \( n \in \bn \),  \( h_n \) can be expressed as a word in \( \mathcal S \) of length \( 26n+12 \). 
\end{proof}


\section{Normal generators}
\label{sec:dilations}

In the final section, given a telescoping 2-manifold \( M \) with a finite number of maximal ends, we establish a class of normal generators for \( H(M) \).  
We then proceed to show that \( H(M) \) is generated by involutions.
The argument presented here is motivated by Ling's proof \cite{LingAlgebraic} of the fact that, given a compact manifold \( N \), the germ at each end of \( N \times \br \) is simple.  

\begin{Def}[Dilation at infinity]
\label{def:dilation}
Let \( M \) be a telescoping 2-manifold, and let \( \mu \) be a maximal end of \( M \).
A homeomorphism \( \sigma \co M \to M \) is  \emph{contracting} (resp., \emph{expanding}) at \( \mu \) if \( \hat\sigma \) fixes \( \mu \)  and there  exists a maximal telescope \( T \) with defining end  \( \mu \) such that  \( \bigcap_{n \in \bn} \sigma^n(T) = \varnothing \)\footnote{This is equivalent to \( \bigcap_{n \in \bn} \hat\sigma^n(\widehat T) = \{\mu\}. \)}  (resp., \( \bigcap_{n\in\bn} \sigma^{-n}(T) = \varnothing \)).
If, in addition, \( T \ssm \sigma(T) \) (resp., \( T \ssm \sigma^{-1}(T) \)) is a tube, then we say that \( \sigma \) is \emph{strongly} contracting (resp., expanding) at \( \mu \).
A homeomorphism is a \emph{(strong) dilation at \( \mu \)} if it is either (strongly) contracting or (strongly) expanding at \( \mu \), and we say, \( \sigma \) is a \emph{(strong) dilation at infinity} if it is a (strong) dilation at each maximal end of \( M \). 
\end{Def}

\textbf{Remark.}
It is not clear to the author if every dilation at infinity is in fact a strong dilation at infinity.  
The author expects this to be true if one assumes the underlying manifold is tame in the sense of Mann--Rafi \cite{MannLarge}; however, we will not investigate this here, as it is tangential to the main point.  

Observe that if \( M = \br^2 \) or \( M = \br^2 \ssm \{0\} \), then the standard notion of a dilation (i.e., a transformation of the form \( x \mapsto kx \) for some \( k \in \br \)) is a dilation at infinity in the sense given above. 
Also observe that if \( M \) is telescoping with two maximal ends, then a dilation at infinity can be contracting (or expanding) at both of the maximal ends. 

Going forward, we will use the following notation for conjugation: 
If \( a \) and \( b \) are elements in a group, we let \( a^b = bab^{-1} \).

\begin{Thm}
\label{thm:dilation2}
Let \( M \) be a telescoping 2-manifold with finitely many maximal ends, and let \( \sigma \in \Homeo(M) \) be a strong dilation.
If \( M \) has a unique maximal end, then every element \( H(M) \) can be expressed as a product of 8 conjugates of \( \sigma \) and \( \sigma^{-1} \); otherwise, every element \( H(M) \) can be expressed as a product of 16 conjugates of \( \sigma \) and \( \sigma^{-1} \)
\end{Thm}

\begin{proof}
Using Corollary~\ref{cor:simple_fragmentation} and the same factorization as in the proof of Lemma~\ref{lem:commutator}, it is enough to prove the following: if \( h \in H(M) \) and \( \{ A_n \}_{n\in \bn} \) is a locally finite family of pairwise disjoint tubes contained in the interior of a maximal telescope \( T \) such that \( h \) is supported in \( A =  \bigcup_{j \in J} A_j \) and such that each component of \( T \ssm A^\mathrm o \) is a tube, then  \( h \) is product of four conjugates of \( \sigma \) and \( \sigma^{-1} \).

Let \( \mu \) be the defining end of \( T \). 
Up to replacing \( \sigma \) with \( \sigma^{-1} \), we may assume that \( \sigma \) is contracting at \( \mu \). 
Therefore, there exists a maximal telescope \( T' \) with defining end \( \mu \) such that \( \bigcap_{n\in\bn} \sigma^n(T') = \varnothing \) and \( T' \ssm \sigma(T') \) is a tube. 
Let \( P_1 = T' \ssm \sigma(T') \), and for \( n \in \bn \), set \( P_n = \sigma^{n-1}(P_1) \). 
Observe that \( T' = \bigcup_{n\in\bn} P_n \). 

Using Theorem~\ref{thm:topology}, we can replace \( h \) with a conjugate so that \( T = T' \) and \( A_n = P_{2n} \) for \( n \in \bn \). 
Just as in the proof of Lemma~\ref{lem:commutator}, we can define the infinite product 
	\[
		\vp := \prod_{k=0}^\infty (\sigma^{2k}\circ h \circ \sigma^{-2k}).
	\]
For \( n \in \bn \), let  \( h_n \) (resp., \( \vp_n \)) be the homeomorphism that agrees with \( h \) (resp., \( \vp \)) on \( P_{2n} \) and restricts to the identity on \( M \ssm P_{2n} \). 
Then, we can write 
	\[ 
		\vp_n = \prod_{k=1}^n h_{n-k+1}^{\sigma^{2k-2}} = h_n \circ h_{n-1}^{\sigma^2} \circ \cdots \circ h_1^{\sigma^{2n-2}}
	\]
and \( \vp = \prod_{n\in\bn} \vp_n \).
Let \( \tau = [\vp,\sigma] \circ [\vp, \sigma]^\sigma \).
We claim that \( h = \tau \) , and hence \( h \) is a product of four conjugates of \( \sigma \) and \( \sigma^{-1} \).
First observe that, as \( \vp \) is supported in \(  \bigcup_{n\in\bn} P_{2n} \),  \( [\vp, \sigma] \) is supported in \( \bigcup_{n=2}^\infty P_n \), and \( [\vp,\sigma]^\sigma \) is supported in \( \bigcup_{n=3}^\infty P_n \), it follows that \( \tau \) is supported in \( \bigcup_{n=2}^\infty P_n \). 
Therefore, we need to check that \( h \) and \( \tau \) agree on \( P_n \) for \( n \in \bn \ssm \{1\} \). 

Note that, for each \( n \in \bn \),  \( h \), \( \vp \), \( [\vp, \sigma] \), \( [\vp, \sigma]^\sigma \), and \( \tau \) each fix \( P_n \) setwise and restrict to the identity on \( \partial P_n \), so we will abuse notation and identify their restrictions to each \( P_n \) with the homeomorphism that agrees with it on \( P_n \) and restricts to the identity elsewhere. 
Observe that
\begin{align*}
[\vp, \sigma] |_{P_{2n}} &= \vp|_{P_{2n}} = \vp_n \\
\intertext{and} 
[\vp, \sigma] |_{P_{2n+1}} &= (\vp^{-1})^\sigma|_{P_{2n+1}} = \left(\vp_n^{-1}\right)^\sigma.
\end{align*}

We break the rest of the argument into two computations depending on whether \( n \) is even or odd.
\begin{align*}
 \tau|_{P_{2n}}	&= \left([\vp, \sigma]\circ [\vp,\sigma]^\sigma\right) |_{P_{2n}} \\ 
 			&= \left([\vp, \sigma]|_{P_{2n}}\right) \circ \left([\vp,\sigma]^{\sigma} |_{P_{2n}}\right) \\
			&= \left([\vp, \sigma]|_{P_{2n}}\right) \circ \left([\vp,\sigma]|_{P_{2n-1}}\right)^\sigma \\
 			&= \vp_n \circ \left(\vp_{n-1}^{-1}\right)^{\sigma^2} \\
 			&= \left(\prod_{k=1}^n h_{n-k+1}^{\sigma^{2k-2}}\right) \circ \left(\prod_{k=1}^{n-1} (h_{n-k}^{-1})^{\sigma^{2k}}\right)\\
			&= h_n\circ \left(\prod_{k=1}^{n-1} h_{n-k}^{\sigma^{2k}}\right) \circ \left(\prod_{k=1}^{n-1} (h_{n-k}^{-1})^{\sigma^{2k}}\right)\\
			&= h_n \\
			&= h|_{P_{2n}} 
\intertext{and} 
\tau|_{P_{2n+1}}	&= \left([\vp, \sigma]\circ [\vp,\sigma]^\sigma\right) |_{P_{2n+1}} \\
				&= \left([\vp, \sigma]|_{P_{2n+1}}\right) \circ \left([\vp,\sigma]^{\sigma} |_{P_{2n+1}}\right) \\
				&= \left([\vp, \sigma]|_{P_{2n+1}}\right) \circ \left([\vp,\sigma]|_{P_{2n}}\right)^\sigma \\
				&= \left(\vp_n^{-1} \right)^\sigma \circ \left( \vp_n \right)^\sigma \\
				&= \mathrm{id}_M \\
				&= h|_{P_{2n+1}}
\end{align*}
Therefore, \( \tau = h \) as desired. 
\end{proof}

\begin{Cor}
\label{cor:dilation}
Let \( M \) be a telescoping 2-manifold with finitely many maximal ends. 
If \( \sigma \in H(M) \) is a dilation at infinity, then \( \sigma \) uniformly normally generates \( H(M) \).
\end{Cor}

\begin{proof}
Let \( \mu \) be a maximal end of \( M \). 
Up to replacing \( \sigma \) by \( \sigma^{-1} \), we can assume \( \sigma \) is contracting at \( \mu \). 
Let \( T \) be a maximal telescope in \( M \) with defining end \( \mu \) such that \( \bigcap_{n\in \bn} \sigma^n(T) = \varnothing \). 
Using Theorem~\ref{thm:topology}, Lemma~\ref{lem:tube_ends}, and the classification of surfaces, there is some \( N_\mu \in \bn \) such that the closure of \( T \ssm \sigma^{N_\mu}(T) \) is a tube; in particular, \( \sigma^{N_\mu} \) is a strong dilation at \( \mu \). 
If \( M \) has a unique maximal end, let \( N = N_\mu \).
Otherwise, let \( \mu' \) be the other maximal end of \( M \),  let \( N_{\mu'} \) be defined analogously, and set \( N = N_\mu N_{\mu'} \).
Then, \( \sigma^N \) is a strong dilation.
By Theorem~\ref{thm:dilation2}, every element of \( H(M) \) is a product of at most \( 16N \) conjugates of \( \sigma \) and \( \sigma^{-1} \); therefore,  \( \sigma \) uniformly generates \( H(M) \). 
\end{proof}

To finish, we prove Corollary~\ref{cor:rotation}.

\begin{proof}[Proof of Corollary~\ref{cor:rotation}]
The case of a perfect set of maximal ends is handled by Theorem~\ref{thm:vlamis}, so we assume \( M \) has two maximal ends.
Let \( P \) be a tube in \( M \).
By Theorem~\ref{thm:topology}(3), we can identify \( M \) with \( P \tilde\times \bz \). 
Using Theorem~\ref{thm:topology}(1), there exists an involution \( \rho \co P \to P \) that swaps the boundary components of \( P \); moreover, if \( M \)---and hence \( P \)---is orientable, \( \rho \) can be chosen to be orientation preserving.  
Define \( \theta \co P \tilde\times \bz \to P \tilde\times \bz \) by \( \theta(x,n) = (\rho(x), -n) \). 
Then, \( \theta \) is an involution of \( P \tilde\times \bz \). 
Let \( \tau \co P \tilde\times \bz \to P \tilde\times \bz \) be given by \( \tau(x,n) = (x, n+1) \). 
Then, it is readily verified that \( \theta^\tau \circ \theta \) is a strong dilation at infinity. 
The result is now a direct application of Theorem~\ref{thm:dilation2}. 
\end{proof}

\bibliographystyle{amsplain}
\bibliography{references}

\end{document}